\documentclass{article}

\usepackage[scale=0.7, a4paper]{geometry}

\usepackage{graphicx} 
\usepackage{epstopdf} 

\usepackage[colorlinks, hypertexnames=false]{hyperref}
\hypersetup{linkcolor=blue, citecolor=red} 

\usepackage{bm}
\usepackage{amsmath} 
\usepackage{amssymb} 
\usepackage{stmaryrd} 
\usepackage{MnSymbol} 
\usepackage{multirow}
\usepackage{enumerate}

\usepackage{amsthm} 
\usepackage{mathtools} 

\usepackage[dvipsnames]{xcolor}

\numberwithin{equation}{section}

\newtheorem{theorem}{Theorem}[section]
\newtheorem{lemma}[theorem]{Lemma}
\newtheorem{remark}[theorem]{Remark}
\newtheorem{corollary}[theorem]{Corollary}

\newcommand{\ds}{\,ds}

\newcommand{\pn}{\,\partial_\mathbf{n}}

\newcommand{\einner}[1]{\left\langle #1\right\rangle}

\newcommand{\norm}[1]{\left\|#1\right\|}

\newcommand{\ddt}{\overline\partial_t}

\title{Finite element approximation to the non-stationary quasi-geostrophic equation}
\author{
     Dohyun Kim\thanks{Hong Kong Centre for
        Cerebro-Cardiovascular Health Engineering, Hong Kong Science Park, Hong Kong SAR, China
        Email: dhkim@hkcoche.org.
    }
    \and
    Amiya K. Pani\thanks{
        Department of Mathematics, Birla Institute of Technology \& Science-Pilani, K.K. Birla
        Goa Campus, NH 17 B, Zuarinagar Goa-403 726, India.
        Email: amiyap@goa.bits-pilani.ac.in.
    }
    \and
    Eun-Jae Park\thanks{
        School of Mathematics and Computing (Computational Science \& Engineering), Yonsei University, Seoul 03722, Korea.
        Email: ejpark@yonsei.ac.kr.
    }
}

\begin{document}

\date{}
\maketitle

\begin{abstract}
    In this paper, $C^1$-conforming element methods are analyzed for the stream function formulation of a single layer non-stationary quasi-geostrophic equation in the ocean circulation model.
    In its first part, some new regularity results are derived, which show exponential decay property when the wind shear stress is zero or exponentially decaying.
    Moreover, when the wind shear stress is independent of time, the existence of an attractor is established.
    In its second part, finite element methods are applied in the spatial direction and for the resulting semi-discrete scheme, the exponential decay property, and the existence of a discrete attractor are proved.
    By introducing an intermediate solution of a discrete linearized problem, optimal error estimates are derived.
     Based on backward-Euler method, a completely discrete scheme is obtained and uniform in  time {\it a priori} estimates are  established.
     Moreover, the existence of a  discrete solution is proved by appealing to a variant of the Brouwer fixed point theorem and then,
      optimal  error estimate is  derived. Finally, several computational experiments with benchmark problems are conducted to confirm our theoretical findings.
\end{abstract}


\pagestyle{myheadings} \thispagestyle{plain}
\markboth{AUTHOR}{$C^1$-conforming FEM for evolutionary QGE}

\section{Introduction}\label{sec:intro}

The stream function formulation of one layer non-stationary quasi-geostrophic (QG) equation in a bounded domain $\Omega\subset \mathbb{R}^2$ with boundary $\partial\Omega$ is to
find the streamfunction $\psi$ defined in the space time domain $\Omega\times(0,\infty)$ such that
\begin{equation}\label{eq:qge}
    -\partial_t\Delta\psi+\nu\Delta^2\psi+J(\psi,\Delta\psi)-\mu\partial_x\psi=\mu F \quad\text{in }\Omega\times(0,\infty),
\end{equation}
with initial condition
\begin{equation}\label{eq:ic}
    \psi(0)=\psi_0\text{ in }\Omega,
\end{equation}
and boundary conditions
\begin{equation}\label{eq:bc}
    \psi=0,\quad\pn\psi=0\text{ on }\partial\Omega\times(0,\infty).
\end{equation}
Here, $J(\varphi,\chi)=\partial_y\varphi\partial_x\chi-\partial_x\varphi\partial_y\chi$,
$\nu$ is the diffusion coefficient, $\mu=1/Ro$ where $Ro$ is the Rossby number.

The QG equation plays an important role in the study of large scale wind-driven oceanic flow \cite{Pedlosky2013,McWilliams2006, Vallis2006}.
Despite the simplicity of the QG equation, the QG equation and its linear variants, such as the Munk equation, preserve many of the important features of the underlying oceanic flows, such as the western boundary currents, and the formation of gyres.
Earlier in  the literature, the problem \eqref{eq:qge} and \eqref{eq:ic} with boundary conditions $\psi=0$ and $\Delta \psi =0$ or periodic boundary conditions are considered.
Those boundary conditions naturally split the problem  into two second-order problems.
Then, the existence, uniqueness, and existence of an attractor (see, \cite{bernier1994, dymnikov, medjo2014, Temam1997}) are studied.
However, for the problem with boundary conditions \eqref{eq:bc}, their analysis breaks down since the problem cannot be  naturally split into two second order problems.

From a numerical point of view, there are several numerical
results regarding wind-driven ocean circulation including QGE
\cite{qge_num1_chekroun2020, FIW, Vernois1966, bryan1963,
kimpanipark21}. In \cite{kim2015, jiang2016, rotundo2016, Kim2018,
kim2020}, B-spline based finite element methods are investigated,
but there remain difficulties for curved domains.
In \cite{FIW}, a $C^1$-conforming finite element method (FEM) is
applied to \eqref{eq:qge} with boundary conditions \eqref{eq:bc}
and \textit{a priori} error estimates are established. Here, while
$H^2$-error estimate is optimal, the error estimates for
$H^1$-norm appears to be suboptimal. Recently, the authors in
\cite{kimpanipark21} considered a nonconforming Morley finite
element method for the stationary QG equation and performed
optimal a priori error analysis along with a posteriori error
estimation.

In this article, we consider $C^1$-conforming FEMs for the
non-stationary QG equation with the no-slip boundary condition,
$\psi=\pn\psi=0$. The analysis presented here can be applied to
any $C^1$-FEM, such as the Argyris element, the Bogner-Fox-Schmitt
(BFS) element, the Hsieh-Clough-Tocher (HCT) element. While
$C^1$-conforming elements are relatively complex to implement,
their high-order continuity allows simple formulation. With a
suitable mapping from the reference element to a physical element,
one can efficiently assemble the global matrix, see
\cite{kirby2019}. Also, compared to nonconforming methods of a
similar order \cite{Kim2018, Kim2016}, $C^1$-FEMs typically yield
smaller degrees of freedom due to their strong inter-element
continuity and they are free from stabilization parameters.

We now briefly summarize the results derived in this article.
\begin{itemize}
    \item [(i)] New regularity results are proved which show exponential decay property
    when the forcing function $F=0$ or exponentially decaying in time.
    When nonzero $F$ is independent of time, the existence of a global attractor is shown.
    \item [(ii)] Based on $C^1$-conforming FEM to discretize in spatial directions, a semidiscrete problem is derived, uniform estimate in time as $t\mapsto \infty$ is proved and the existence of a discrete attractor is shown.
    \item [(iii)] Using elliptic projection, optimal error estimates in $L^{\infty}(H^j),\;j=0,1$ and $L^{\infty}(L^{\infty})$-norms
    are established.
    \item [(iv)] Based on the backward Euler method, a completely discrete scheme is derived and existence of a unique discrete solution is proved using an uniform estimate in time of  Dirichlet norm.  A part from the existence of the discrete attractor,  \textit{a priori} error estimates are derived.
 \item [(v)]  Finally, some numerical experiments on benchmark problems are conducted and results confirm our theoretical  findings.
\end{itemize}
When either $F=0$ or $F$ decays exponentially, it is observed that all the results derived in this paper including error estimates
have exponential decay properties.

In  a recent related paper \cite{AMNS}, authors have discussed $C^1$-conforming VEM for the nonstationary Navier-Stokes equations in stream-function vorticity formulation and derived optimal error estimate in $H^1$ norm under smallness
assumption on the data and using fixed point arguments. We remark here that the present analysis can be extended  to include VEM  with some appropriate modifications.

Throughout this article, we use the standard notation of Lebesgue spaces $L^p$ and Sobolev spaces $W^{m,p}(\Omega)$ with their respective norms $\|\cdot\|_{0,p} $ and $\|\cdot\|_{m,p}$.
Further, when $p=2$, we apply the standard notation for Hilbert spaces like $L^2$ and $H^m(\Omega)$ with, respective, inner-products $(\cdot,\cdot)$ and $(\cdot,\cdot)_m$, norms $\|\cdot\|$ and $\|\cdot\|_m$.
Moreover, $|\cdot|_m$ is denotes by the seminorm.
For a Banach space $X$ with norm $\|\cdot\|_{X}$, let $L^p(0,\infty;X)$ or simply $L^p(X)$ whenever there is no confusion.

The rest of the paper is organized as follows.
In the next section, we derive some new regularity results and the existence of an attractor.
In Section~\ref{sec:fem}, $C^1$-conforming FEMs are introduced.
In Section~\ref{sec:err}, \textit{a priori} error estimates are derived for $L^2$-, $H^1$-, and $H^2$-norms.
Then, the backward Euler method is applied to derive the fully discrete scheme in Section~\ref{sec:back} with \textit{a priori} error estimates in both space and time.
Some numerical experiments are provided in Section~\ref{sec:num} and concluding remarks are given in Section~\ref{sec:con}.

\section{A priori estimates}\label{sec:regularity}

The variational formulation that we shall use in this paper is to seek $\psi(t)\in H^2_0(\Omega))$ such that
for almost all $t>0$
\begin{equation}\label{eq:weakform}
    (\partial_t(\nabla\psi),\nabla \chi)+\nu a(\psi,\chi) 
    +b(\psi;\psi,\chi) + \mu b_0(\psi, \chi)
    =\mu(F,\chi)   \forall\chi\in H^2_0(\Omega)
\end{equation}
with $\psi(0)=\psi_0$, where the bilinear forms and the trilinear form are 
\begin{equation*}
    \begin{aligned}
        a(v,w)&:=(\Delta v, \Delta w),\;\;b_0(v,w):=-\frac{1}{2}\left[(v_x,w)-(v,w_x)\right],\;{\mbox {and}}\;
        b(\psi;v,w):=(\Delta\psi, v_yw_x-v_xw_y).
    \end{aligned}
\end{equation*}
This section deals with \textit{a priori} bounds, estimates of an attractor, and some regularity results of the solution to \eqref{eq:weakform}.

The trilinear form $b(\cdot;\cdot,\cdot)$ satisfies
\begin{align*}
    \text{(i) }   & b(\psi,\varphi,\varphi)=0,\quad\forall \Delta\psi\in L^2(\Omega),\varphi\in H^1(\Omega). \\
    \text{(ii) }  & b(\psi;\omega,v)=-b(\psi;v,\omega).\\
    \text{(iii) } & b(\psi;\varphi,z)\leq M\left\{
    \begin{aligned}
         & \norm{\Delta\psi}\norm{\nabla\varphi}_{0,4}\norm{\nabla z}_{0,4}
         &                                                                  & \forall \Delta\psi\in L^2(\Omega),\varphi,z\in W^{1,4}(\Omega),                     \\
         & \norm{\Delta\psi}\norm{\Delta\varphi}\norm{\Delta z}
         &                                                                  & \forall \Delta\psi,\Delta\varphi,\Delta z\in L^2(\Omega),                           \\
         & \norm{\varphi}_{1,\infty}\norm{\Delta\psi}\norm{\nabla z}
         &                                                                  & \forall\varphi\in W^{1,\infty}(\Omega),\Delta\psi\in L^2(\Omega),z\in H^1_0(\Omega).
    \end{aligned}
    \right.
\end{align*}
Note that for $v$ or $w$ in $H^1_0(\Omega)$, $b_0(v,w) = - (v_x, w)$ and for $w=v$,\; $b_0(v,v)=0$.

The Ladyzhenskaya inequalities in two dimension are  given by:
\begin{enumerate}
    \item For $\varphi\in H^1_0(\Omega)$:
          \begin{equation*}
              \norm{\varphi}_{0,4}\leq C_L\norm{\varphi}^{1/2}\norm{\nabla\varphi}^{1/2}.
          \end{equation*}
    \item For $\varphi\in H^2_0(\Omega)$:
          \begin{equation*}
              \norm{\nabla\varphi}_{0,4}\leq C_L\norm{\nabla\varphi}^{1/2}\norm{\Delta\varphi}^{1/2}.
          \end{equation*}
\end{enumerate}
Since $v\in H^1_0(\Omega)$,
\begin{equation*}
    \norm{v}\leq \frac{1}{\sqrt{\bar{\lambda}_{1}}}\norm{\nabla v},
\end{equation*}
where $\bar{\lambda}_1>0$ is the minimum eigenvalue of the eigenvalue problem $-\Delta \phi = \bar{\lambda} \phi$ in $\Omega$ with homogeneous Dirichlet boundary condition.
Moreover, if $v\in H^2_0(\Omega),$ then
\begin{equation}\label{eq:eigen_H1}
    \norm{\nabla v}\leq \frac{1}{\sqrt{\bar{\lambda}_{1} }}\norm{\Delta v}.
\end{equation}

For our subsequent use, we need the following inequality for $\varphi\in H^2_0(\Omega){\cap H^3(\Omega)}$
\begin{equation*}
    \norm{\Delta \varphi}\leq \frac{1}{\sqrt {\bar{\lambda}_1}}\norm{\nabla\Delta\varphi}.
\end{equation*}
Since this proof is nonstandard, we sketch its proof.
Note that using integration by parts
$$\norm{\Delta \varphi}^2= -(\nabla \psi, \nabla\Delta \varphi) \leq \norm{\nabla \varphi} \norm{\nabla\Delta \varphi} \leq \frac{1}{\sqrt{\bar{\lambda}_1}}\norm{\Delta \varphi}\norm{\nabla\Delta \varphi},$$
and the required result follows.
\begin{lemma}\label{lem:qge_prioribound}
    Assume that $\psi_0\in H^1_0(\Omega)$ and $F\in L^2(H^{-2})$.
    Then for $0<\alpha< (\nu\bar{\lambda}_1)/{2}$, the following holds
    \begin{align*}
        &\norm{\nabla\psi(t)}^2+\beta e^{-2\alpha t}\int_0^t e^{2\alpha s}\norm{\Delta\psi(s)}^2\ds 
        \leq e^{-2\alpha t}\left(\norm{\nabla\psi_0}^2+\frac{\mu^2}{\nu}\int_0^te^{2\alpha s}\norm{F}_{-2}^2\ds\right),
    \end{align*}
    where $\beta=(\nu-(2\alpha/{\bar{\lambda}_1}))>0$.
    Moreover, there holds
    \begin{equation}\label{eq:energy_alpha0}
        \|\nabla\psi(t)\|^2+\nu\int_0^t\|\Delta\psi(s)\|^2\ds\leq \|\nabla\psi_0\|^2+\frac{\mu^2}{\nu}\|F\|_{L^2(H^{-2})}^2.
    \end{equation}
\end{lemma}
\begin{proof}
    Choose $\chi=e^{2\alpha t}\psi$ in \eqref{eq:weakform} and observe that $b(\cdot;\psi,\psi)=0$ with $b_0(\psi,\psi)=0$.
    Then, we arrive at
    \begin{align*}
        \frac{1}{2}\frac{d}{dt}(e^{2\alpha t}\norm{\nabla\psi(t)}^2)+\nu e^{2\alpha t}\norm{\Delta\psi}^2
        &-\alpha e^{2\alpha t}\norm{\nabla\psi}^2 \leq e^{2\alpha t}\mu\norm{F}_{-2}\norm{\Delta\psi}\\
        &\leq \frac{\mu^2}{2\nu}e^{2\alpha t}\norm{F}_{-2}^2+\frac{\nu}{2}e^{2\alpha t}\norm{\Delta\psi}^2.
    \end{align*}
    A use of \eqref{eq:eigen_H1} with kick-back arguments, and then after integration with respect to time as well as multiplying by $e^{-2\alpha t}$ yield the desired result.
    Finally, \eqref{eq:energy_alpha0} follows with $\alpha=0$.
\end{proof}
\begin{remark}\label{remark:2.2}
    To discuss the dynamics of the problem \eqref{eq:qge}--\eqref{eq:bc}, the forcing function $F$ is considered to be independent of time.
    Then, one obtains
    \begin{equation*}
         \norm{\nabla\psi(t)}^2+\beta e^{-2\alpha t}\int_0^t e^{2\alpha s}\norm{\Delta\psi(s)}^2\ds\leq e^{-2\alpha t}\norm{\nabla\psi_0}^2+\frac{\mu^2}{2\nu\alpha}(1-e^{-2\alpha t})\norm{F}_{-2}^2.
     \end{equation*}
\end{remark}
\begin{theorem}[Wellposedness]\label{thm:qge_wellposedness}
    Let $\psi_0\in H^1_0(\Omega)$ and $F\in L^2(0,T;H^{-2}(\Omega))$.
    Then, for any $T>0$, the problem \eqref{eq:qge}--\eqref{eq:bc} admits a unique weak solution $\psi \in C^0(0,T,H^1_0(\Omega))\cap L^2(0,T;H^2_0(\Omega))\cap H^1(0,T;H^{-1}(\Omega))$ satisfying
    \begin{equation}\label{eq:weakform2}
        <\partial_t(\nabla\psi),\nabla\chi>+\nu a(\psi,\chi) +b(\psi;\psi,\chi)+ \mu b_0 (\psi,\chi)=(\mu F,\chi)\quad\forall\chi\in H^2_0(\Omega),\text{ a.e. }t\in(0,T)
    \end{equation}
    with $\psi(0)=\psi_0$.
    Moreover, the solution $\psi$ depends continuously on $\psi_0$.
\end{theorem}
Bernier \cite{bernier1994} has discussed existence of a unique strong solution to problem \eqref{eq:qge} -\eqref{eq:ic} with boundary condition $\psi=0$ and $\Delta\psi=0$.
Using the Faedo-Galerkin method, below, we briefly indicate the proof of Theorem~\ref{thm:qge_wellposedness}.
\begin{proof}
    Using the Faedo-Galerkin method and the \textit{a priori} uniform bound as in Lemma~\ref{lem:qge_prioribound} for the Galerkin approximations, standard weak and weak* compactness arguments with the Aubin-Lions compactness argument, see \cite{aubin1963}, one easily provides the proof of the existence of a solution.
    It, therefore, remains to prove the continuous dependence property.
    Let $\psi_i$, $i=1,2$ be weak solutions of \eqref{eq:weakform2} with initial data $\psi_{i,0}\in H^1_0$.
    Setting $\Phi=\psi_1-\psi_2$, it follows that $\Phi$ satisfies
    \begin{equation*}
        (\partial_t(\nabla\Phi),\nabla\chi)+\nu(\Delta\Phi,\Delta\chi)+\mu b_0(\Phi,\chi)=-b(\Phi;\psi_1,\chi)-b(\psi_2;\Phi,\chi).
    \end{equation*}
    Choose $\chi=\Phi$ and notice that $b(\psi_2;\Phi,\Phi)=0$ with $b_0(\Phi,\Phi)=0$.
    With the aid of Lemma~\ref{lem:qge_prioribound}, the Ladyzhenskaya inequality and Young's inequality: $ab\leq a^p/p+b^q/q$ with $p=4/3$ and $q=4$, we arrive at
    \begin{align*}
        \frac{1}{2}\frac{d}{dt}\norm{\nabla\Phi}^2+\nu\norm{\Delta\Phi}^2
        & \leq M\norm{\Delta\Phi}\norm{\nabla\psi_1}_{0,4}\norm{\nabla\Phi}_{0,4}                                       \\
        & \leq C_L^2M\norm{\Delta\Phi}^{3/2}\norm{\nabla\psi_1}^{1/2}\norm{\Delta\psi_1}^{1/2}\norm{\nabla\Phi}^{1/2}   \\
        & \leq \frac{3\nu}{4}\norm{\Delta\Phi}^2+C\nu^{-3}\norm{\nabla\psi_1}^2\norm{\Delta\psi_1}^2\norm{\nabla\Phi}^2 \\
        & \leq \frac{3\nu}{4}\norm{\Delta\Phi}^2+C\nu^{-3}(\|\nabla\psi_{0,1}\|^2+\frac{\mu^2}{\nu}\|F\|_{L^2(H^{-2})}^2)\norm{\Delta\psi_1}^2\|\nabla\Phi\|^2.
    \end{align*}
    Integrate with respect to time and use Lemma~\ref{lem:qge_prioribound} to obtain
    \begin{equation*}
        \norm{\nabla\Phi}^2
        \leq \norm{\nabla\Phi(0)}^2
        +{C(\|\nabla\psi_{0,1},\|F\|_{L^2(H^{-2})},\mu^2,\nu^{-3})}
        \int_0^t\norm{\Delta\psi_1}^2\norm{\nabla\Phi}^2\ds.
    \end{equation*}
    An application of Gronwall's Lemma with the bound from Lemma~\ref{lem:qge_prioribound} yields
    \begin{equation*}
        \begin{aligned}
        \norm{\nabla\Phi(t)}^2
        &\leq C\exp\Big(C\int_0^t\norm{\Delta\psi_1}^2\ds\Big)\norm{\nabla\Phi(0)}^2\\
        &\leq Ce^{CT}\norm{\nabla\Phi(0)}^2.
        \end{aligned}
    \end{equation*}
    where generic constants $C$ depend on $\|\nabla\psi_{0,1}\|$, $\|F\|_{L^2(H^{-2})}$, $\mu^2$ and $\nu^{-3}$.
    As a consequence, the uniqueness holds and this completes the rest of the proof.
\end{proof}
\begin{remark}\label{remark:2.4}
    The L'Hospital rule implies
    \begin{equation*}
        \limsup_{t\rightarrow\infty} e^{-2\alpha t}\int_0^t e^{2\alpha s}\norm{\Delta\psi(s)}^2\ds
        =\frac{1}{2\alpha}\limsup_{t\rightarrow\infty}\norm{\Delta\psi(t)}^2.
    \end{equation*}
    Thus, by taking limit supremum in Remark~\ref{remark:2.2} with $\alpha=\nu\bar{\lambda}_1/4$, we obtain
    \begin{equation*}
        \limsup_{t\rightarrow\infty}\norm{\Delta\psi(t)}\leq \frac{\mu^2}{2\nu^2}\norm{F}_{-2}.
    \end{equation*}
    Moreover, for $\alpha=0$ and time-independent $F$, we obtain
    \begin{equation*}
        \frac{d}{dt}\norm{\nabla\psi(t)}^2+\nu\norm{\Delta\psi(t)}^2\leq \frac{\mu^2}{\nu}\norm{F}_{-2}^2.
    \end{equation*}
    Integrating both sides from $t$ to $t+T_0$ for some $T_0>0$,
    \begin{equation*}
        \begin{aligned}
          \norm{\nabla \psi(t+T_0)}^2+ \nu\int_t^{t+T_0}\norm{\Delta\psi}^2\ds
            &\leq \norm{\nabla\psi(t)}^2+\frac{T_0\mu^2}{\nu}\norm{F}_{-2}^2\\
            &\leq \norm{\psi_0}^2+\frac{\mu^2}{\alpha\nu}\norm{F}_{-2}+\frac{\mu^2 T_0}{\nu}\norm{F}_{-2}^2\\
            &\leq\norm{\nabla\psi_0}^2+\frac{\mu^2}{\nu}(\frac{1}{\alpha}+T_0)\norm{F}_{-2}^2.
        \end{aligned}
    \end{equation*}
    Hence, dropping the first nonnegative term and then taking limit supremum, we arrive at
    \begin{equation*}
        \nu\limsup_{t\rightarrow\infty}\int_t^{t+T_0}\norm{\Delta\psi}^2\ds\leq \norm{\nabla\psi_0}^2+\frac{\mu^2}{\nu}(\frac{1}{\alpha}+T_0)\norm{F}_{-2}^2.
    \end{equation*}
\end{remark}

\subsection{Regularity results.} This subsection focusses on regularity results to be used in our subsequent analysis.

The following lemma which deals with the regularity of the biharmonic problem with Dirichlet data \cite{blum1980} is useful in proving the regularity result of this subsection.

\begin{lemma} [Regularity of biharmonic equation] \label{lem:regularity-biharmonic}
    For a polygonal domain $\Omega$ with Lipschitz boundary and given function $g\in H^{\delta}(\Omega)$ with $\delta\in [1/2,1)$, let $\Phi\in H^2_0(\Omega)$ solves $\Delta^2 \Phi= g$ in $\Omega$, then
    $$\|\Phi\|_{H^{2+\delta}(\Omega)} \leq C_R\|g\|_{H^{-2+\delta}(\Omega)}.$$
    The above result is also valid for the elliptic index $\delta=1$, if, in addition, the domain $\Omega$ is convex.

    Moreover, if the domain $\Omega$ is convex with all interior angles less than $126.283^\circ$, then there holds for $\delta\in (1,2]$
    $$\|\Phi\|_{H^{2+\delta}(\Omega)} \leq C_R\|g\|_{H^{-2+\delta}(\Omega)}.$$
\end{lemma}

\begin{theorem}[Regularity]\label{thm:qge_regularity-1}
    Given $\psi_0\in H^2_0(\Omega)\cap H^3(\Omega)$ and $F\in H^{-1}(\Omega)$, the following estimate holds for all $t >0$
    \begin{equation*}
        \|\partial_t(\nabla \psi_m)(t)\|^2 + \beta e^{-2\alpha t}\int_{0}^{t}e^{2\alpha s}\|\partial_t \Delta {\psi}_m(s)\|^2 \ds \leq C,
    \end{equation*}
    where $C$ is a positive constant depending on $\|\psi_0\|_3, \|F(0)\|_{-1}, \|F_t\|_{L^2(H^{-2})}$, $\mu$ and $\nu$.
\end{theorem}
\begin{proof}
    Let $\{\lambda_j\}_{j=1}^{\infty}$ be eigenvalues and $\{\phi_j\}_{j=1}^{\infty}$ be corresponding orthonormalized eigenvectors of $\Delta^2 \phi =\lambda \phi$ in $\Omega$ with homogeneous Dirichlet boundary conditions.
    The set of eigenvectors forms a basis in $H^2_0$, $H^1_0$ and $L^2$.
    Setting $V_m= \overline{\text{span}\{\phi_1,\cdots,\phi_m\}}$, define $\psi_m(t) := \sum_{j=1}^{m} \alpha_j(t)\phi_j \in V_m$ as a solution of
    \begin{equation}\label{eq:Galerkin}
        (\partial_t(\nabla\psi_m),\nabla \phi_k)+\nu a(\psi_m,\phi_k)
        +b(\psi_m;\psi_m,\phi_k) + \mu b_0(\psi_m, \phi_k)
        =(\mu F,\phi_k), \quad k=1,\cdots,m
    \end{equation}
    with $\psi_m(0)=\sum_{j=1}^{m} (\psi_0,\phi_j) \phi_j$.
    Differentiate with respect to time \eqref{eq:Galerkin} and obtain
    \begin{equation}\label{eq:diff_Galerkin}
        \begin{aligned}
            &(\partial_{tt}(\nabla\psi_m),\nabla \phi_k)
            +\nu a(\partial_t\psi_m,\phi_k)
            + \mu b_0(\partial_t \psi_m, \phi_k)\\
            &\quad+b(\partial_t \psi_m;\psi_m,\phi_k) +b(\psi_m;\partial_t \psi_m,\phi_k)
            =(\mu \partial_t F,\phi_k), \quad k=1,\cdots,m.
        \end{aligned}
    \end{equation}
    Multiply \eqref{eq:diff_Galerkin} by $\alpha'_k$, then sum up form from $k=1$ to $m$ to arrive at
    \begin{equation*}
        \begin{aligned}
            &(\partial_{tt}(\nabla\psi_m),\partial_t(\nabla\psi_m))+\nu a(\partial_t\psi_m,\partial_t\psi_m)+\mu b_0(\partial_t\psi_m,\partial_t\psi_m)\\
            &\quad+b(\partial_t\psi_m;\psi_m,\partial_t\psi_m)+b(\psi_m;\partial_t\psi_m,\partial_t\psi_m)=(\mu\partial_tF,\partial_t\psi_m).
        \end{aligned}
    \end{equation*}
   An application of  skew-symmetric property of $b(\psi_m;\cdot,\cdot)$ and $b_0(\cdot,\cdot)$ yields
    \begin{equation}\label{eq:estimate-1}
        \frac{1}{2} \frac{d}{dt} \|\partial_t(\nabla \psi_m)(t)\|^2 + \nu \|\partial_t\Delta \psi_m(t)\|^2
        =b(\partial_t \psi_m;\partial_t\psi_m, \psi_m) + \mu(\partial_t F, \partial_t\psi_m).
    \end{equation}
    To estimate the first term on the right hand side of \eqref{eq:estimate-1},
    a use of the generalized H\"older inequality with the Ladyzhenskaya inequality, boundedness of $\|\nabla \psi_m\|$
    and  the Young's inequality $ab \leq (a/p)+ (b/q)$ with $p=4/3$, $q=4$, $a=\big((\sqrt{\nu/4})\|\partial_t \Delta \psi_m\| \big)^{3/2} $ and $b=\big((\sqrt{4/\nu})\|\partial_t\nabla \psi_m\|\|\nabla \psi_m\|\| \Delta \psi_m\|\big)^{1/2}$
    show
    \begin{align}\label{eq:b-estimate-1}
        b(\partial_t \psi_m; \partial_t \psi_m, \psi_m)
        &\leq M \|\partial_t \Delta \psi_m\|\|\partial_t \nabla \psi_m\|_{0,4}\|\nabla \psi_m\|_{0,4}\nonumber\\
        &\leq C \|\partial_t \Delta \psi_m\|^{3/2}\|\partial_t\nabla \psi_m\|^{1/2}\|\nabla \psi_m\|^{1/2}\| \Delta \psi_m\|^{1/2}\nonumber\\
        &\leq \frac{C}{\nu}\|\Delta \psi_m\|^2\|\partial_t \nabla \psi_m\|^2 + \frac{\nu}{4}\|\partial_t\Delta \psi_m\|^2.
    \end{align}
    Moreover, an application of the Young's inequality implies
    \begin{equation*}
        \mu(\partial_t F, \partial_t \psi_m) \leq \frac{\mu^2}{\nu} \|\partial_t F\|_{-2}^2 + \frac{\nu}{4}\|\partial_t \Delta \psi_m\|^2.
    \end{equation*}
    Substitute \eqref{eq:b-estimate-1} in \eqref{eq:estimate-1}. Then, multiply by
    $e^{2\alpha t}$ and  use kickback arguments to arrive at
    \begin{equation*}
        \begin{aligned}
        \frac{d}{dt} \| e^{\alpha t} \partial_t(\nabla {\psi}_m)(t)\|^2
        &+\nu e^{2\alpha t} \|\partial_t\Delta {\psi}_m(t)\|^2
        - 2 \alpha e^{2\alpha t} \|\partial_t(\nabla {\psi}_m) (t)\|^2\\
        &\leq \frac{2\mu^2}{\nu} e^{2\alpha t} \|{\partial_t F}\|_{-2}^2 + \frac{C}{\nu} \|\Delta \psi_m(t)\|^2 \|e^{\alpha t} \partial_t (\nabla \psi_m)(t)\|^2.
        \end{aligned}
    \end{equation*}
    By \eqref{eq:eigen_H1}, $\|\nabla v\|^2 \leq (1/\bar{\lambda}_1) \|\Delta v\|^2$, we obtain after integration with respect to time
    \begin{equation*}
        \begin{aligned}
        &\|e^{\alpha t}\partial_t(\nabla \psi_m)(t)\|^2 + \beta \int_{0}^{t} e^{2\alpha s}\|\partial_t \Delta {\psi}_m(s)\|^2 \ds\\
        &\leq \| \partial_t(\nabla \psi_m)(0)\|^2+ \frac{2\mu^2}{\nu} \int_{0}^t e^{2\alpha s} \|{\partial_t F}\|_{-2}^2\ds \\
        &\quad+ \frac{C}{\nu} \int_{0}^{t} \|\Delta \psi_m(s)\|^2 \|e^{\alpha t} \partial_t(\nabla \psi_m)(s)\|^2\ds.
        \end{aligned}
    \end{equation*}
    An application of the Gronwall's Lemma with multiplication by $e^{-2\alpha t}$ shows
    \begin{equation}\label{eq:estimate-4}
        \begin{aligned}
        &\|\partial_t(\nabla \psi_m)(t)\|^2 + \beta e^{-2\alpha t}\int_{0}^{t} e^{2\alpha s} \|\partial_t \Delta {\psi}_m(s)\|^2 \ds\\
        &\quad\leq \Big(e^{-2\alpha t} \|\partial_t(\nabla \psi_m)(0)\|^2 +\frac{2\mu^2}{\nu} \|{\partial_t F}\|_{L^2(H^{-2})}^2 \Big)
         \exp \Big( C \nu^{-1}\int_{0}^{t} \|\Delta \psi_m\|^2\ds\Big).
        \end{aligned}
    \end{equation}
    It remains to show the estimate $\|\partial_t(\nabla \psi_m)(0)\|$, we note using
    \begin{equation*}
        \|J(\psi_m(0),\Delta \psi_m(0))\|_{-1} \leq C \|\Delta \psi_m(0)\|_{0,4}\|\nabla \psi_m(0)\|_{0,4} \leq C \|\psi_m(0)\|^2_{3}\leq C \|\psi_0\|_{3}^2
    \end{equation*}
    and
    \begin{equation*}
        \|\Delta^2\psi_m(0)\|_{-1} \leq C \|\psi_m(0)\|_{3}\leq C \|\psi_0\|_{3}
    \end{equation*}
    that
    \begin{align}
        \|\partial_t(\nabla \psi_m)(0)\| &\leq \|\partial_t(\Delta \psi_m)(0)\|_{-1} \nonumber\\
        &\leq \nu \|\Delta^2 \psi_m(0)\|_1 + \|J(\psi_m,\Delta\psi_m)(0)\|_{-1} + \mu (\|\partial_x \psi_m(0)\|_{-1} + \|F(0)\|_{-1})\nonumber\\
        &\leq C \Big(\nu \|\psi_0\|_3 + \|\psi_0\|_3^2 + \mu(\|\nabla\psi_0\| + \|F(0)\|_{-1})\Big).
    \end{align}
    A use of \eqref{eq:energy_alpha0} of Lemma~\ref{lem:qge_prioribound}, which is also valid for the Galerkin approximation 
    to \eqref{eq:estimate-4} with $\|\Delta \psi_m(0)\| \leq \|\Delta \psi_0\|$ yields 
    \begin{equation*}
        \begin{aligned}
        \|\partial_t(\nabla \psi_m)(t)\|^2&+ \beta e^{-2\alpha t}\int_{0}^{t} e^{2\alpha s} \|\partial_t \Delta {\psi}_m(s)\|^2 \ds\\
        &\leq  C(\mu,\nu, \|\psi_0\|_3, \|F(0)\|_{-1}, \|F_t\|_{L^2(H^{-2})}). 
        \end{aligned}
    \end{equation*}
    Taking limit as $m \rightarrow \infty$, we obtain for all $t\in (0,T_0]$
    \begin{equation*}
        \begin{aligned}
        \|\partial_t(\nabla \psi)(t)\|^2&+ \beta e^{-2\alpha t}\int_{0}^{t} e^{2\alpha s} \|\partial_t \Delta {\psi} (s)\|^2 \ds\\
        &\leq  C(\mu,\nu, \|\psi_0\|_3, \|F(0)\|_{-1}, \|F_t\|_{L^2(H^{-2})}). 
        \end{aligned}
    \end{equation*}
    This concludes the rest of the proof.
\end{proof}
Hence forth, the arguments and estimates given in this subsection are formal and these can be justified rigorously following the usual Galerkin type procedure as in Theorem~\ref{thm:qge_regularity-1}, and then passing to the limit.
\begin{theorem}\label{thm:qge_regularity-2}
    Let $\psi_0\in H^{2+\delta}(\Omega)\cap H^2_0(\Omega)$ and $F, F_t \in L^2(H^{-2})$ with $F(0)\in H^{-1}(\Omega)$.
    Then, there is a positive constant $C$ depending on $\|\psi_0\|_3, \|F_t\|_{L^2(H^{-2})}, \|F(0)\|, \mu$ and $\nu$ such that for all $t>0$ for $1/2<\delta \leq 1$
    \begin{equation*}
        \| \psi (t)\|_{2+\delta} + e^{-2\alpha t} \int_{0}^t e^{2\alpha s} \| \psi (s)\|_{2+\delta}^2\ds
        \leq C.
    \end{equation*}
    Moreover, if we further assume that $F_t \in L^2(\Omega),$ then for $1< \delta \leq 2$
    \begin{equation}
       e^{-2\alpha t} \int_{0}^t e^{2\alpha s} \| \psi (s)\|_{2+\delta}^2\ds \leq C.
    \end{equation}
\end{theorem}
\begin{proof}
    Note that from the main equation \eqref{eq:qge} with
    \begin{equation}\label{eq:psi-0}
        \nu \Delta^2 \psi(t) = g(t):= \partial_t \Delta \psi - J (\psi,\Delta \psi) + \mu (\psi_x + F),
    \end{equation}
    and elliptic regularity result
    \begin{align}\label{eq:negative-estimate}
        \nu \|\psi(t)\|_{2+\delta} \leq \nu \|\Delta^2 \psi(t)\|_{-2+\delta}
        &\leq \|\partial_t \Delta \psi(t)\|_{{-2+\delta}} + \|J(\psi,\Delta \psi)\|_{{-2+\delta}}\nonumber\\
        &+\mu \Big(\|\psi_x\|_{{-2+\delta}} + \|F(t)\|_{{-2+\delta}}\Big)\nonumber\\
        &\leq C \|\partial_t \nabla \psi(t)\| + \|J(\psi,\Delta \psi)\|_{{-2+\delta}}
        + \mu\Big(\|\nabla \psi\| + \|F(t)\|_{-2} \Big).
    \end{align}
    In order to estimate the second term on the right hand side of \eqref{eq:negative-estimate}, recall the definition of $J$.
    After we use generalized H\"older's inequality with $1/2+1/p+1/q=1$ where $p=2/\sigma$ and $q=2/(1-\sigma)$.
    \begin{align*}
        |(J(\psi,\Delta \psi), \varphi)(t)|
        &= |b (\psi(t);\psi(t),\varphi)|\\
        &\leq M  \|\Delta \psi(t)\|\|\psi(t)\|_{1,p} \|\varphi\|_{1,q}.
    \end{align*}
    Finally, the Sobolev imbedding theorem $H^1\subset L^p$ and $H^\sigma\subset L^q$ shows
    \begin{equation*}
        |(J(\psi,\Delta\psi,\varphi)(t))|\leq C \|\Delta \psi(t)\|^2 \|\varphi\|_{H^{1+\sigma}(\Omega)}.
    \end{equation*}
    Note that $J(\psi,\Delta \psi)(t) \in H^{-(1+\sigma)}(\Omega)$ and hence, $g(t)\in H^{-2+(1-\sigma)}(\Omega)$.
    Then, from elliptic regularity Lemma~\ref{lem:regularity-biharmonic}, we obtain $\psi(t) \in {H^{3-\sigma}(\Omega)}$.
    Since $H^{3-\sigma}$ with $0<\sigma <1/2$ is continuously embedded in $C^1(\bar{\Omega})$, there  holds
    \begin{equation*}
        |(J(\psi,\Delta \psi), \varphi)(t)|:= |b (\psi(t);\psi(t),\varphi)|
        \leq M  \|\Delta \psi(t)\|\|\psi(t)\|_{1,\infty} \|\varphi\|_{1}.
    \end{equation*}
    This implies $J(\psi,\Delta \psi)(t) \in H^{-1}(\Omega)$.
    Therefore, $g(t)\in H^{-1}(\Omega)$.
    Now, an application of elliptic regularity Lemma~\ref{lem:regularity-biharmonic} shows $\psi(t)\in H^3(\Omega).$

    For the estimate \eqref{eq:psi-0}, a use of the Theorem~\ref{thm:qge_regularity-1} with elliptic regularity yields
    \begin{align*}
        \beta e^{-2\alpha t} \int_{0}^t \| \psi(s)\|_{2+\delta}^2\ds &\leq C \beta e^{-2\alpha t} \int_{0}^t \|\Delta^2 \psi(s)\|^2\ds \nonumber\\
        &\leq C e^{-2\alpha t} \int_{0}^t \Big( \|\Delta \psi_t(s)\|^2 + \|J(\psi; \Delta\psi)(s)\|^2 + \mu^2 \|\psi_x\|^2 + \mu^2 \|F\|^2\Big)\ds\nonumber\\
        &\leq C \big( \alpha,\nu,\mu,\|\psi_0\|_3, \|F\|_{L^2(L^2)}, \|F_t\|_{L^2(H^{-1})}\big).
    \end{align*}
    This completes the rest of the proof.
\end{proof}

The following theorem focusses on the regularity results which will be needed in our error analysis.
\begin{theorem}\label{thm:qge_regularity-3}
    Let $\psi_0\in H^4(\Omega)\cap H^2_0(\Omega)$ and $F_t \in L^2(H^{-1}(\Omega)).$ Then there is a positive constant $C$ depending on $\|\psi_0\|_4, \|F_t\|_{-1}, \|F(0)\|, \mu$ and $\nu$ such that for all $t>0$ and for $1<\delta \leq 2$
    \begin{equation}
        \nu \Big(\|\psi(t)\|_{2+\delta}+ \|\partial_t \Delta \psi(t)\|^2 \Big)+ e^{-2\alpha t} \int_{0}^t e^{2\alpha s} \|\nabla \partial_t^2\psi(s)\|^2\ds \leq C.
    \end{equation}
Moreover, for $1/2<\delta \leq 1$
    \begin{equation*}
        \nu e^{-2\alpha t} \int_{0}^t e^{2\alpha s} \|\partial_t \psi(s)\|^2_{2+\delta}\ds \leq C.
    \end{equation*}
\end{theorem}
\begin{proof}
    Differentiate equation \eqref{eq:qge} with respect to time.
    After multiplying $e^{2\alpha t} \partial_t^2 \psi$ to the both sides and integrate over $\Omega$, we arrive at
    \begin{align}\label{eq:psi-t-estimate-2-1}
        e^{2\alpha t} \|\partial_t^2\nabla \psi(t)\|^2 &+ \frac{\nu}{2} \frac{d}{dt} \Big( e^{2\alpha t} \|\partial_t \Delta {\psi}(t)\|^2\Big)= \alpha e^{2\alpha t} \|\partial_t \Delta \psi(t)\|^2\nonumber\\
        &- e^{\alpha t}  \big( b(\partial_t\psi;\psi, e^{\alpha t} \partial_t^2 \psi)
        + b(\psi;\partial_t\psi, e^{\alpha t} \partial_t^2 \psi) \big)\nonumber\\
        &- \mu  b_0( e^{\alpha t} {\partial_t \psi}, e^{\alpha t} \partial_t^2 \psi) + \mu   (e^{\alpha t} \partial_t {F}, e^{\alpha t} \partial_t^2\psi).
    \end{align}
    The second term on the right hand side is bounded by
    \begin{align} \label{eq:b-tt-1}
       - e^{\alpha t}  \big( b(\partial_t \psi;\psi, e^{\alpha t} \partial_t^2 \psi)
        + b(\psi;\partial_t \psi, e^{\alpha t} \partial_t^2 \psi) \big)\leq C \|\partial_t \Delta \psi \|^2  \|\psi\|^2_3 + \frac{1}{6}  e^{2\alpha t} \|\partial_t^2 \nabla \psi(t)\|^2.
    \end{align}
    For the last two terms are bounded by
    \begin{align}\label{eq:b-0-F-t-1}
        - \mu  \Big( b_0( e^{\alpha t} {\partial_t \psi}, e^{\alpha t} \partial_t^2\psi)
       &-(e^{\alpha t}\partial_t {F}, e^{\alpha t} \partial_t^2 \psi)\Big)\nonumber\\
       &\leq \frac{3\mu^2}{\nu}   e^{2\alpha t} \big( \|\partial_t \nabla \psi \|^2 + \|\partial_t F\|^2_{-1}\big)
       + \frac{\nu}{3} e^{2\alpha t} \|\partial_t^2 \nabla\psi\|^2.
    \end{align}
    On substitution of \eqref{eq:b-tt-1} and \eqref{eq:b-0-F-t-1} in \eqref{eq:psi-t-estimate-2-1} with kick back argument, then integrate with respect to time with Theorems~\ref{thm:qge_regularity-1}--\ref{thm:qge_regularity-2} with $\|\Delta \psi_t(0)\| \leq C (\|\psi_0\|_{4})$ yields
    \begin{align}
        &e^{-2\alpha t} \int_{0}^t e^{2\alpha s} \|\partial_t^2 \nabla \psi (s)\|^2 \ds + \nu \|\partial_t \Delta {\psi} (t)\|^2 \leq \nu \|\partial_t \Delta \psi (0\|^2 \nonumber\\
        &\quad+ C(\alpha,\nu,\mu)  e^{-2\alpha t} \int_{0}^t e^{2\alpha s} \Big( \|\partial_t \Delta \psi (s)\|^2 \big(1+\|\psi\|^2_3\big)
        + \|\partial_t \nabla \psi\|^2 + \|\partial_t F\|^2_{-1}\Big)\ds.
    \end{align}
    As in the proof of Theorem~\ref{thm:qge_regularity-2}, using the estimates of the Theorem~\ref{thm:qge_regularity-1}, we complete the rest of the proof.
\end{proof}

The last theorem of this subsection deals on the regularity results to be used subsequently.
\begin{theorem}\label{thm:qge_regularity-4}
    Let $\psi_0\in H^4(\Omega)\cap H^2_0(\Omega)$, $F_t \in H^{-1}(\Omega)$, and $F_{tt}\in H^{-2}(\Omega)$.
    Then there is a positive constant $C$ depending on $\|\psi_0\|_4,\; \|F_t\|_{-1},\; \|F_{tt}\|_{-2},\; \|F(0)\|,\; \mu$ and $\nu$ such that for all $t>0$ and for $1/2<\delta \leq 1$
    \begin{equation}\label{eq:qge_regularity-estimate-3}
        \nu \tau(t)\Big(\|\partial_t \psi(t)\|_{2+\delta}^2+ \|\partial^2_{t} \nabla \psi(t)\|^2 \Big)+ e^{-2\alpha t} \int_{0}^t e^{2\alpha s}\tau(s) \|\partial_{t}^2\Delta \psi (s)\|^2\ds \leq C,
    \end{equation}
    where $\tau(t)= \min \{t,1\}.$ Moreover, for $1\leq \delta \leq 2$
    \begin{equation*}
        \nu e^{-2\alpha t} \int_{0}^t e^{2\alpha s}\tau(s) \|\partial_{t}^2 \psi(s)\|^2_{2+\delta}\ds \leq C.
    \end{equation*}
\end{theorem}
\begin{proof}
    On differentiating the equation \eqref{eq:qge} with respect to time twice and then forming an inner-product with $-\tau(t) e^{2\alpha t} \partial^2_t \psi $, using the property $(i)$ of the trilinear form $b(\cdot;\cdot,\cdot)$ and property of $b_0(\cdot,\cdot)$, it now follows that
    \begin{align*}
        &\frac{1}{2} \frac{d}{dt} \Big(\tau(t) e^{2\alpha t} \|\partial_t^2 \nabla \psi (t)\|^2\Big) + {\nu}  \tau(t) e^{2\alpha t} \|\partial_t^2\Delta {\psi} (t)\|^2 = (\alpha \tau(t) + \tau'(t)) e^{2\alpha t} \|\partial_t^2 \nabla \psi (t)\|^2 \nonumber\\
        &\quad - \tau(t) e^{\alpha t} \big( b(\partial_t^2\psi;\psi, e^{\alpha t} \partial^2_t \psi)
        + 2 b(\partial_t\psi;\partial_t \psi, e^{\alpha t} \partial^2_t\psi) \big)\nonumber\\
        &\quad + \mu   \tau(t) (e^{\alpha t} \partial_t^2 {F}, e^{\alpha t} \partial_t^2 \psi).
    \end{align*}
    For the first term on the right hand side of \eqref{eq:psi-t-estimate-2-1}, note that $\tau(t) \leq 1,$ and $\tau'\leq 1.$ The second term on the right hand side is bounded by
    \begin{align} \label{eq:b-tt}
        -\tau(t)e^{\alpha t} \big( b(\partial_t^2\psi;\psi, e^{\alpha t} \partial^2_t \psi)
            &+ 2 b(\partial_t\psi;\partial_t \psi, e^{\alpha t} \partial^2_t\psi)\big) \leq \frac{1}{3} \tau(t) e^{2\alpha t} \|\partial_t^2 \Delta \psi (t)\|^2\nonumber\\
            &+ C e^{2\alpha t} \big(\|\Delta \psi_t\|^2 \|\psi\|^2_3 + \|\partial_t^2 \nabla \psi\|^2+ \|\partial_t \psi\|_3\big).
    \end{align}
    The third term is bounded by
    \begin{align}\label{eq:b-0-F-t}
        \mu\tau(t)\einner{e^{\alpha t} \partial_t^2 {F}, e^{\alpha t} \partial_t^2 \psi} \leq \frac{3\mu^2}{2\nu}\tau(t) e^{2\alpha t}\|\partial_t^2F\|^2_{-2}
        + \frac{\nu}{6} \tau(t) e^{2\alpha t}\|\partial_t^2\Delta\psi\|^2.
    \end{align}
    On substitution of \eqref{eq:b-tt} and \eqref{eq:b-0-F-t} in \eqref{eq:psi-t-estimate-2-1}, use kick back argument and integration with respect to time with multiplication of $e^{-2\alpha t}$ to obtain
    \begin{align}
        \tau(t) \|\partial_t^2 \nabla \psi(t)\|^2 &+ e^{-2\alpha t} \nu \int_{0}^t \tau(s) e^{2\alpha s}
        \|\partial_t^2 \Delta \psi(s)\|^2 \ds 
        \leq 2(\alpha + 1)e^{-2\alpha t} \int_{0}^{t} e^{2\alpha s} \|\partial_t^2 \nabla \psi(s)\|^2\ds\nonumber\\
        &+ C(\alpha,\nu,\mu)  e^{-2\alpha t} \int_{0}^t e^{2\alpha s} \Big( \|\partial_t^2\nabla \psi (s)\|^2 \|\psi\|_3^2 
        + \|\partial_t \Delta \psi \|^2  \|\partial_t\psi\|^2_3\Big)\ds\nonumber\\
        &+C \big(\frac{\mu^2}{\nu}\big)  e^{-2\alpha t} \int_{0}^t e^{2\alpha s} \|\partial_t^2 F(s)\|^2_{-2} \ds.
    \end{align}
    A use of the estimates in Theorems~\ref{thm:qge_regularity-1}--\ref{thm:qge_regularity-3}
completes the proof of the estimate \eqref{eq:qge_regularity-estimate-3}.
Proceed in a similar manner as in the proof of the Theorem~\ref{thm:qge_regularity-2} to complete the rest of the proof.
\end{proof}


\begin{remark}
    The results of Theorem~\ref{thm:qge_regularity-4} hold for $\tau=1$ under higher regularity on the initial data $\psi_0$, that is, $\psi_0\in H^6(\Omega) \cap H^2_0(\Omega)$
    under some compatibility conditions.
\end{remark}
\begin{remark}\label{rmk:decaying}
    If $F=0$, then the following regularity results hold
    \begin{align*}
        \norm{\psi (t)}_{2+\delta}&\leq Ce^{-\alpha t},\\
        \int_0^te^{2\alpha t}(\norm{\partial_t\psi (t)}_{2+\delta}^2+\norm{\psi (t)}_{2+\delta})&\leq C.
    \end{align*}
    %
    If $F$ decays exponentially, then also the exponential decay property holds. 
    Moreover, if $F\in L^2(0,T;H^{-1}(\Omega))$ and $\partial_tF\in L^2(0,T;H^{-2}(\Omega))$, then error is bounded for all $t> 0$.
\end{remark}
\subsection{Absorbing set and global attractor.}
This subsection is on the study of the dynamics of the system \eqref{eq:qge}--\eqref{eq:bc} under the assumption that $F$ is time independent.

Let $S(t)$ for $t\geq0$ be the solution operator form which takes $\psi_0\in H^1_0(\Omega)$ into $\psi(t)$.
This family $\{S(t)\}_{t\geq0}$ forms a semigroup of operators on $H^1_0(\Omega)$.
Below, we discuss the main result of this subsection.
\begin{lemma}\label{lem:attractor-H1}
    With $H=H^1_0(\Omega),$ 
there exists $\rho_0>0$ such that the ball $B_H(0,\rho_0$ 
is absorbing in $H$ 
in the sense that for $\rho>0$, there exists $t^\star(\rho)>0$ such that for $t\geq t^\star(\rho)$, $S(t)B_H(0,\rho)\subset B_H(0,\rho_0).$ 
    %
\end{lemma}
\begin{proof}
    Let $\psi_0\in B_H(0,\rho_0)$.
Then, from Lemma~\ref{lem:qge_prioribound}, we obtain
    \begin{equation*}
        \norm{\nabla\psi(t)}^2\leq e^{-2\alpha t}\norm{\nabla \psi_0}^2+\frac{\mu^2}{2\nu\alpha}\norm{F}_{-2}^2(1-e^{-2\alpha t}).
    \end{equation*}
    With $\rho^2=\frac{\mu^2}{\nu\alpha}\norm{F}_{-2}^2$, it follows that
    \begin{equation*}
        \norm{\nabla\psi(t)}^2\leq e^{-2\alpha t}(\rho_0^2-\frac{1}{2}\rho^2)+\frac{1}{2}\rho^2.
    \end{equation*}
    Choosing
    \begin{equation*}
        t\geq\frac{1}{\alpha}\log\left(\frac{2\rho_0^2-\rho^2}{\rho^2}\right)=t^\star(\rho).
    \end{equation*}
    We now obtain $\norm{\nabla\psi(t)}^2\leq \rho^2$, that is, $S(t)B_H(0,\rho)\subset B_H(0,\rho_0)$.
    For $\rho_0^2\leq \rho/2$, the result follows trivially for all $t>0$.
    Therefore, $B_H(0,\rho_0)$ is an absorbing set in $H$.
    This proves the theorem.
\end{proof}

\section{Finite element method}\label{sec:fem}
In this section, $C^1$-conforming FEM is applied to problem \eqref{eq:weakform}.
Then, for the corresponding semi-discrete section, we discuss the existence of a discrete attractor.
Let $\mathcal{T}_h$ be a shape regular partition of $\overline{\Omega}$ where $h=\max{K\in\mathcal{T}_h}h_K$ with $h_K=\text{diam}(K)$.
A general $C^1$-conforming FE space is defined by
\begin{equation*}
    S_h\subset C^1(\Omega)\cap\mathcal{P},\quad S_h^0=\{v\in S_h:\pn v= v=0\text{ on }\partial\Omega\}.
\end{equation*}
where $\mathcal{P}$ is a piecewise polynomial space.
We say $S_h$ is of degree $k$ if $S_h$ consists of piecewise polynomial up to degree $k\geq 3$.
\begin{remark}
    Each of $C^1$-finite element space leads to different definition for $\mathcal{T}_h$, $\mathcal{P}$, and $k$.
    \begin{itemize}
        \item When $\mathcal{T}_h$ is a triangular partition and $\mathcal{P}$ is a piecewise cubic polynomial space over a collection of sub-triangles of $\mathcal{T}_h$ with $k=3$, $S_h$ is a Hsieh-Clough-Tocher (HCT) element.
        \item When $\mathcal{T}_h$ is a triangular partition and $\mathcal{P}$ is piecewise quintic polynomial space over $\mathcal{T}_h$ with $k=5$, $S_h$ is an Argyris element.
        \item When $\mathcal{T}_h$ is a rectangular partition and $\mathcal{P}$ is piecewise bi-cubic polynomial space over $\mathcal{T}_h$ with $k=3$, $S_h$ is a Bogner-Fox-Schmit (BFS) element.
    \end{itemize}
\end{remark}
Denote $I_h:H^2_0(\Omega)\rightarrow S_h^0$ as the interpolation operator with the following property:
For $v\in H^{2+\delta}(\Omega)\cap H^2_0(\Omega)$, there is a positive constant $C$, independent of $h$, such that for $\delta >1/2$
\begin{equation}\label{estimate:interpolation}
    \norm{v-I_hv}_j\leq Ch^{\min(k+1,2+\delta)-j}\norm{v}_{2+\delta},\ j=0,1,2.
\end{equation}


The semi-discrete problem after applying the $C^1$-conforming FEM is to seek $\psi_h(t)\in S_h^0$ such that
\begin{equation}\label{eq:discprob}
    (\partial_t\nabla\psi_h,\nabla\chi)+\nu a(\psi_h,\chi)+b(\psi_h;\psi_h,\chi)+\mu b_0(\psi_h,\chi)=\mu (F,\chi)\quad \forall\chi\in S_h^0,\;t\in(0,T],
\end{equation}
with $\psi_h(0)=\psi_{0,h}\in S_h^0$ to be defined later.
Since $S_h^0$ is finite dimensional, \eqref{eq:discprob} leads to a system of nonlinear ODEs.
As $b(\cdot;\cdot,\cdot)$ is locally Lipschitz function, an application of Picard's theorem yields the existence of a unique local discrete solution $\psi_h(t)$ for $t\in(0,t_h^\star)$ for some $t_h^\star>0$.
For applying continuation argument so that solution $\psi_h(t)$ can be continued for all $t\in (0,T]$, we need to derive an \textit{a priori} bound on $\psi_h(t)$ for all $t\in (0,T]$.
\begin{lemma}[\textit{A priori} bound]\label{lem:discapriori}
    Let $\psi_{0,h}\in S_h^0$ be the initial condition for the discrete problem such that $\norm{\nabla \psi_{0,h}} \leq C\norm{\nabla \psi_0}$.
Then, for all $t\in (0,T]$ and for $0<\alpha<\frac{\nu\lambda_1}{2}$, the following estimate holds:
    \begin{equation}\label{eq:estimate-semi-discrete-1}
        \norm{\nabla\psi_h(t)}^2+ \beta e^{-2\alpha t}\int_0^te^{2\alpha t}\norm{\Delta\psi_h(s)}^2\ds\leq C e^{-2\alpha t}\norm{\nabla \psi_{0}}^2+\frac{\mu^2}{\nu}\norm{F}_{L^2(H^{-2})}^2, 
    \end{equation}
    where $\beta=(\nu-\frac{2\alpha}{\lambda_1})>0$.
    Moreover,
    \begin{align}
        \label{eq:estimate-semi-discrete-2}
        \mathop {\lim \sup}_{t\rightarrow \infty}\|\Delta \psi_h(t) \|\leq \frac{\mu^2}{\nu }
        \|F\|_{L^2(H^{-2})}.
    \end{align}

\end{lemma}

\begin{proof}
    Choose $\chi=\psi_h$ in \eqref{eq:discprob}.
Since $b(\psi_h;\psi_h,\psi_h)=0$ and $b_0(\psi_h,\psi_h)=0$, multiply $e^{2\alpha t}$ to the both sides to arrive at
    \begin{equation*}
        \frac{d}{dt}(e^{2\alpha t}\norm{\nabla\psi_h(t)}^2)+2 (\nu-\frac{\alpha}{\lambda_1})e^{2\alpha t}\norm{\Delta\psi_h(t)}^2\leq \frac{\mu^2}{\nu}e^{2\alpha t}\norm{F}_{-2}^2+ \nu \norm{\Delta\psi_h(t)}^2.
    \end{equation*}
    Use the kick-back argument and then integrate from $0$ to $t$ and multiply $e^{2\alpha t}$ with the bound of $\norm{\nabla\psi_{0,h}}$ to complete the estimate \eqref{eq:estimate-semi-discrete-1}.

    In order to estimate \eqref{eq:estimate-semi-discrete-2}, first drop the first term on the left hand side as it is nonnegative and then take the limit superior.
An application of L'Hospital rule yields the desired estimate and this concludes the rest of the proof.
\end{proof}
When $F$ is time independent, following the proof of Lemma~\ref{lem:discapriori}, we easily obtain
\begin{equation*}
\norm{\nabla\psi_h(t)}^2+ \beta e^{-2\alpha t}\int_0^te^{2\alpha t}\norm{\Delta\psi_h(s)}^2\ds\leq C e^{-2\alpha t}\norm{\nabla \psi_{0,h}}^2+\frac{\mu^2}{2\alpha \nu}\norm{F}_{-2}^2 (1-e^{-2\alpha t}).
    \end{equation*}
Therefore following the continuous case,
we derive the following result on the existence of a discrete absorbing set.
\begin{corollary}\label{cor:discabsorbing}
    There exists a bounded absorbing ball $B_\rho(0)=\{v_h\in S_h^0:\norm{\nabla v_h}\leq \rho\}$ where $\rho$ is given by $\rho^2=\frac{\mu^2}{\alpha \nu}\norm{F}_{-2}^2$.
    %
\end{corollary}

For the semi-discrete solution, assume that the following estimate hold:
\begin{equation*}
    e^{-2\alpha t}
    \int_{0}^t e^{2\alpha s}
    \|\nabla \partial_t^2 \psi_h(s)\|^2 \ds \leq C.
\end{equation*}
This proof will be on the lines of the regularity results proved for the continuous problem, and it involves an introduction of discrete biharmonic operator with analysis quite tedious, therefore, we refrain from giving a proof of it.

\section{Error estimates for the semi-discrete problem.}\label{sec:err}
This section focusses on optimal error estimates in $L^\infty(H^1)$-norm, for the semi-discrete problem \eqref{eq:discprob}.

\subsection{Elliptic projection.}
In order to obtain optimal order of convergence, we split the total error $\psi-\psi_h$ as
\begin{equation*}
    \psi-\psi_h=(\psi-\widetilde{\psi}_h)-(\psi_h-\widetilde{\psi}_h)=:\eta-\zeta
\end{equation*}
where $\widetilde{\psi}_h\in S_h^0$ is an elliptic projection,
which is defined as follows:
Given $\psi$, find $\widetilde{\psi}_h\in S_h^0$ satisfying
\begin{equation}\label{eq:auxProb}
    A(\psi;\psi-\widetilde{\psi}_h,\chi)=0\quad\forall\chi\in S_h^0
\end{equation}
where $A(\cdot;\varphi,\chi)$ is a linearized operator given by
\begin{equation*}
    A(\psi;\varphi,\chi)=\nu(\Delta\varphi,\Delta\chi)+b(\psi;\varphi,\chi)+b(\varphi;\psi,\chi) +\mu b_0(\varphi,\chi)+\lambda(\nabla\varphi,\nabla\chi)
\end{equation*}
for a large positive $\lambda$ to be defined later.
Note that $A(\cdot;\cdot,\cdot)$ satisfies boundedness
\begin{equation}
    |A(\psi;\varphi,w)|\leq M|\psi|_2|\varphi|_2|w|_2\quad\forall \psi,\varphi, w\in H^2_0(\Omega)\label{eq:discconti}
\end{equation}
and coercive with sufficiently large $\lambda>0$.
For $\varphi\in H^2_0(\Omega)$, we have
\begin{equation*}
    \begin{aligned}
        A(\psi;\varphi,\varphi)
         & =\nu|\varphi|_2^2+b(\varphi;\psi,\varphi)+\lambda\norm{\nabla\varphi}^2 \\
         & \geq \nu|\varphi|_2^2-M\norm{\psi}_{1,\infty}|\varphi|_2\norm{\nabla\varphi}+\lambda\norm{\nabla\varphi}^2 \\
         & \geq \frac{\nu}{2}|\varphi|_2^2+(\lambda-C(M,\norm{\psi}_{1,\infty}))\norm{\nabla\varphi}^2.
    \end{aligned}
\end{equation*}
Choose large enough $\lambda>0$ so that $(\lambda-C(M,\norm{\psi}_{1,\infty}))>0$.
Then with $\nu_0=\nu/2$
\begin{equation}\label{eq:disccoer}
    A(\psi;\varphi,\varphi)\geq \nu_0|\varphi|_2^2.
\end{equation}
Thus, for a given $\psi$, the Lax-Milgram theorem yields the existence of a unique $\widetilde{\psi}_h\in S_h^0$.
Below, we discuss estimates of $\eta=\psi-\widetilde{\psi}_h$.
\begin{lemma}\label{lem:apriori_eta}
    There holds for $\delta\in(1/2,1]$
    \begin{equation*}
        \norm{\eta}_0 + \norm{\eta}_1 + h^\delta|\eta|_2\leq Ch^{2\delta}\norm{\psi}_{2+\delta},
    \end{equation*}
    for $\delta \in (1,2)$
    \begin{equation*}
        \norm{\eta}_0 + h \norm{\eta}_1 + h^\delta|\eta|_2\leq Ch^{2\delta}\norm{\psi}_{2+\delta},
    \end{equation*}
    and for $\delta\geq2$
    \begin{equation*}
        \norm{\eta}_j\leq Ch^{\min(k+1,2+\delta)-j}\norm{\psi}_{2+\delta},\;j=0,1,2.
    \end{equation*}
\end{lemma}

\begin{proof}
    From the coercivity \eqref{eq:disccoer}, the boundedness \eqref{eq:discconti}, and the definition of $\eta$ \eqref{eq:auxProb}, we arrive at
    \begin{equation*}
        \nu_0|\eta|_2^2\leq A(\psi;\eta,\eta)=A(\psi;\eta,\psi-I_h\psi)\leq M|\eta|_2|\psi-I_h\psi|_2,
    \end{equation*}
    and hence, using interpolation property \eqref{estimate:interpolation}, it follows for all $\delta > 1/2$ that
    \begin{equation}\label{estimate:eta-H-2}
        |\eta|_2 \leq C(\nu,M)  h^{\min(k+1,2+\delta)-2}  |\psi|_{2+\delta}.
    \end{equation}

    We now appeal to the Aubin-Nitsche duality argument.
    For given $g\in H^{-1}(\Omega)$, $\Phi^\star\in H^2_0(\Omega)$ solves the adjoint problem
    \begin{equation}\label{eq:aubindual}
        A(\psi;w,\Phi^\star)=\einner{w,g}   \forall w\in H^2_0(\Omega).
    \end{equation}
    with the
    following elliptic regularity for $\delta\in (1/2, 1]$
    \begin{equation}\label{estimate:regularity-0}
        \norm{\Phi^\star}_{2+\delta}\leq C_{reg}\norm{g}_{-1}.
    \end{equation}

    Setting $w=\eta$ and $g=-\Delta \eta$ in \eqref{eq:aubindual} and it follows for $\delta\in(1/2,1]$
    \begin{equation*}
        \begin{aligned}
            \norm{\eta}_1^2 & =A(\psi;\eta,\Phi^\star-I_h\Phi^\star) \\
                            & \leq M|\psi|_2|\eta|_2|\Phi^\star-I_h\Phi^\star|_2
        \end{aligned}
    \end{equation*}
    Hence, a use of the elliptic regularity results for $\delta\in(1/2,1]$ with interpolation error estimate \eqref{estimate:interpolation} yields
    \begin{equation*}
        \norm{\eta}_1\leq Ch^\delta|\psi|_2|\eta|_2,
    \end{equation*}
    and $L^2$ estimate follows from $\|\eta\| \leq \|\eta\|_1.$

    In addition, assume that
    the adjoint problem satisfies the following elliptic regularity result for $\delta\in (1,2]$
    \begin{equation}\label{eq:dual_highreg}
        \norm{\Phi^\star}_{2+\delta} \leq C_{reg}\norm{g}_{\delta-2}.
    \end{equation}

    In order to estimate in lower norm, we note that
    \begin{equation}\label{eq:def-2-delta}
        \|\eta\|_{2-\delta} = \sup_{0\neq g\in H^{2-\delta}(\Omega)} \frac{| \einner{\eta, q}|}{\|g\|_{\delta-2}}.
    \end{equation}
    Choose $w=\eta$ in \eqref{eq:aubindual} to obtain
    \begin{align}
        \einner{\eta, g} & =A(\psi;\eta,\Phi^\star- I_h\Phi^\star)\nonumber \\
                         & \leq M|\psi|_2 |\eta|_2 |\Phi^\star-I_h\Phi^\star|_2\label{eq:aubinL2} \\
                         & \leq Ch^\delta  |\psi|_2 |\eta|_2  |\Phi^{\star}|_{2+\delta} \nonumber \\
                         & \leq Ch^\delta |\psi|_2 |\eta|_2 \norm{g}_{\delta-2}.\nonumber
    \end{align}
    A use of \eqref{eq:def-2-delta} in \eqref{eq:aubinL2} shows for $\delta\in (1,2]$
    \begin{equation}\label{estimate-eta-2-delta}
        \|\eta\|_{2-\delta} \leq C h^{\delta}  |\eta|_2,
    \end{equation}
    and the rest of the estimate follows.
This completes the proof.
\end{proof}

A similar results hold for the estimate of $\partial_t\eta$.
\begin{lemma}
    There is a positive constant $C>0$ such that for $\delta\in (1/2,1]$
    \begin{equation}\label{eq:estimate-eta-t-1}
        \norm{\partial_t\eta}+h^\delta|\eta|_2\leq Ch^{2\delta}(\norm{\psi}_{2+\delta}+\norm{\partial_t\psi}_{2+\delta}),
    \end{equation}
    for $\delta \in (1,2)$
    \begin{equation}\label{eq:estimate-eta-t-2}
        \norm{\partial_t\eta}+ h \norm{\partial_t \eta}_1+h^\delta|\eta|_2\leq Ch^{2\delta}(\norm{\psi}_{2+\delta}+\norm{\partial_t\psi}_{2+\delta}),
    \end{equation}
    and for $\delta\geq 2$
    \begin{equation}\label{eq:estimate-eta-t-3}
        \norm{\partial_t\eta}_j\leq Ch^{\min(k+1,2+\delta)-j}(\norm{\psi}_{2+\delta} + \norm{\partial_t\psi}_{2+\delta}),\ j=0,1,2.
    \end{equation}
\end{lemma}
\begin{proof}
    Differentiate \eqref{eq:auxProb} with respect to time to arrive at
    \begin{equation*}
        A(\psi;\partial_t\eta,\chi)=-A(\partial_t\psi;\eta,\chi).
    \end{equation*}
    Now we repeat the argument in Lemma~\ref{lem:apriori_eta} to complete the rest of the proof.
\end{proof}

\subsection{\textit{A priori} error estimates.}
In this subsection, optimal error estimates are derived.

Since $(\psi-\psi_h)(t)=\eta(t)-\zeta(t)$, triangle inequality yields
\begin{equation*}
    \norm{\nabla(\psi-\psi_h)(t)}\leq \norm{\nabla\eta}+\norm{\nabla\zeta}.
\end{equation*}
As the estimate of $\|\nabla\eta\|$ is known from Lemma~\ref{lem:apriori_eta},
it is enough to estimate $\|\nabla\zeta\|$.
From \eqref{eq:weakform}, \eqref{eq:discprob} and the property \eqref{eq:auxProb}, we obtain
\begin{equation}\label{eq:zeta_eq}
    \begin{aligned}
    (\partial_t(\nabla\zeta),\nabla\chi)&+\nu (\Delta\zeta,\Delta\chi)+ b(\psi;\zeta,\chi) + \mu b_0(\zeta, \chi)\\
    &=(\partial_t(\nabla\eta),\nabla\chi) - \lambda(\nabla\eta,\nabla\chi)+I_1(\chi),
    \end{aligned}
\end{equation}
where
\begin{equation*}
    \begin{aligned}
        I_1(\chi) & =b(\psi;\psi_h,\chi)+b(\tilde{\psi}_h;\psi,\chi)-b(\psi;\psi,\chi)-b(\psi_h;\psi_h,\chi) \\
                  & =-b(\eta;\eta,\chi)+b(\eta;\zeta,\chi)-b(\zeta;\psi_h,\chi)\\
                  & = -b(\eta;\eta,\chi)-b(\zeta;\psi_h,\chi)
    \end{aligned}
\end{equation*}
Note $(\partial_t(\nabla\eta),\nabla\chi)=(\partial_t\eta,-\Delta\chi)$ and $(\nabla\eta,\nabla\chi)=(\eta,-\Delta\chi)$ as $\chi \in S_h^0$.
In the remainder of this paper, we denote $\hat{f}(t)=e^{\alpha t}f(t)$
\begin{lemma}\label{lem:apriori_zeta_H1}
    There holds
    \begin{equation} \label{zeta-1}
        \norm{\nabla\zeta(t)}^2\leq C e^{-2\alpha t}  \Big(\norm{\nabla \zeta(0)}^2+
        C \int_0^te^{2\alpha s}(\norm{\partial_t\eta}^2+\norm{\eta}^2)\ds\Big)
        \exp \big({C(M,\nu)\int_0^t|\psi|_2^2\ds}\big).
    \end{equation}
\end{lemma}
\begin{proof}
    Substitute $\chi=e^{2\alpha t}\zeta$ in \eqref{eq:zeta_eq} and use $(\nabla\eta,\nabla\chi)=(\eta,-\Delta\chi)$ with the  Poincar\`e inequality and  $\hat{\zeta}(t)=e^{\alpha t}\zeta(t)$ to obtain
    \begin{equation}\label{auxbound_for_zeta_H1}
        \frac{d}{dt}\norm{\nabla\hat\zeta(t)}^2+2(\nu-\frac{2\alpha}{\lambda_1})|\hat\zeta|_2^2\leq 2(e^{\alpha t}\norm{\partial_t\eta}+\norm{\hat\eta})|\hat\zeta|_2+2e^{2\alpha t}I_1(\zeta).
    \end{equation}
    To estimate the last term on the right-hand side, apply the Ladyzhenskaya's inequality with the Young's inequality to arrive at
     \begin{eqnarray} \label{I-2}
          2e^{2\alpha t}I_1(\zeta)
            & =&2e^{-\alpha t}(b(\hat\zeta;\hat\eta,\hat\zeta)-b(\hat\eta;\hat\eta,\hat\zeta)-b(\hat\zeta;\hat\psi,\hat\zeta))\nonumber\\
             & \leq & 2Me^{-\alpha t}(|\hat\zeta|_2^{3/2}\norm{\nabla\hat\eta}^{1/2}|\hat\eta|_2^{1/2}+|\hat\eta|_2^{3/2}\norm{\nabla\hat\eta}^{1/2}|\hat\zeta|_2+|\hat\zeta|_2^{3/2}\norm{\nabla\hat\psi}^{1/2}|\hat\psi|_2^{1/2}\norm{\nabla\hat\zeta}^{1/2}) \nonumber\\
             & \leq & \frac{\nu}{4}|\hat\zeta|_2^2+C(M,\nu) \Big(e^{-4\alpha t}(\norm{\nabla\hat\eta}^2|\hat\eta|_2^2+|\hat\eta|_2^3\norm{\nabla\hat\eta})+e^{-2\alpha t}\norm{\nabla\psi}^2|\hat\psi|_2^2\norm{\nabla\hat\zeta}^2\Big).
    \end{eqnarray}
 Moreover,  for the first term on the right hand side of (\ref{zeta-1}),  a use of the Young's inequality shows
 \begin{equation} \label{eta-t}
 2(e^{\alpha t}\norm{\partial_t\eta}+\norm{\hat\eta})|\hat\zeta|_2 \leq \frac{\nu}{4}|\hat\zeta|_2^2 + C \big( e^{2\alpha t} \norm{\partial_t \eta}^2 + \norm{\hat\eta})^2\big).
 \end{equation}
 On substitution of (\ref{I-2}) and \eqref{eta-t} in \eqref{auxbound_for_zeta_H1}, an application of  the Gronwall Lemma yields
    \begin{equation*}
        \norm{\nabla\hat\zeta(t)}^2\leq e^{C(M,\nu)\int_0^te^{-2\alpha t}\norm{\nabla\psi}^2|\hat\psi(s)|_2^2\ds}\left(\norm{\nabla\zeta(0)}^2+C(M,\nu)\int_0^te^{2\alpha s}(\norm{\partial_t\eta}^2+\norm{\eta}^2)\ds\right).
    \end{equation*}
    For the integral term appeared in the exponential, it follows using boundedness of $\norm{\nabla\psi(s)}$ that
    \begin{equation*}
        \int_0^te^{-2\alpha s}\norm{\nabla\psi(s)}^2|\hat\psi|_2^2\ds\leq C\int_0^t|\psi(s)|_2^2\ds
    \end{equation*}
    This
    completes the rest of the proof.
\end{proof}

Finally, using triangle inequality with Lemma~\ref{lem:apriori_zeta_H1} and \eqref{eq:auxProb}, we derive the main theorem on error analysis.
\begin{theorem}\label {thm:apriori}
    Let $\psi$ and $\psi_h$, respectively, be the solution of \eqref{eq:weakform} and \eqref{eq:discprob}.
    Then, there exists a positive constant $C$ independent of $h$ such that
    \begin{equation*}
        \begin{aligned}
            \norm{\nabla(\psi-\psi_h)(t)}^2 & \leq Ch^{2r}\norm{\psi(t)}_{2+\delta}^2 \\
                                            & + e^{-2\alpha t+\hat{C}}h^{2r} \left(\norm{\psi_0}_{2+\delta}^2+ \int_0^te^{2\alpha s}(\norm{\psi}_{2+\delta}^2+\norm{\partial_t\psi}_{2+\delta}^2)\ds\right),
        \end{aligned}
    \end{equation*}
    provided $\psi_h(0)\in S_h^0$ is chosen as an interpolant with property
    \begin{equation*}
        \norm{\psi_0-\psi_h(0)}_1\leq Ch^{r}\norm{\psi_0}_{2+\delta}.
    \end{equation*}
    Here,
    \begin{equation*}
        r=\left\{\begin{array}{ll}
            2\delta, \quad & \textrm{if } 1/2<\delta < 1, \\
            \min(k+1,2\delta)-1, & \textrm{if } 1\leq \delta < 2, \\
            \min(k+1,2+2\delta)-1,& \textrm{if } 2\geq \delta,
        \end{array}
        \right.
    \end{equation*}
    and
    \begin{equation*}
        \hat{C}=C(M,\nu)\int_0^t|\psi|_2^2\ds.
    \end{equation*}
\end{theorem}

As a consequence of Lemma \ref{lem:apriori_zeta_H1}, if we choose ${\psi}_h(0)$ as elliptic projection at $t=0$, then
$\zeta(0)=0$.
With higher regularity, that is, $\delta \geq 1$, there holds the following superconvergence result:
\begin{equation} \label{eq:supercgt}
    \norm{\nabla\zeta(t)}^2\leq C e^{-2\alpha t+\hat{C}}h^{2(r+1)} \Big (\int_0^te^{2\delta s}(\norm{\partial_t\psi}_{2+\delta}^2+\norm{\psi}_{2+\delta}^2)\ds\Big).
\end{equation}

Application of the triangle inequality then yields the following optimal error estimate in $L^{\infty}(0,T;L^2).$
\begin{corollary}\label{cor:optimal-L2}
    Let $\psi$ and $\psi_h$, respectively, be the solution of \eqref{eq:weakform} and \eqref{eq:discprob}.
    Then, the following inequality holds for $\delta \geq 1,$
    \begin{equation*}
        \begin{aligned}
            \norm{(\psi-\psi_h)(t)}^2 & \leq Ch^{2(r+1)}  e^{-2\alpha t+\hat{C}}\left( \int_0^te^{2\alpha s}(\norm{\psi}_{2+\delta}^2+\norm{\partial_t\psi}_{2+\delta}^2)\ds\right),
        \end{aligned}
    \end{equation*}
    provided $\psi_h(0)\in S_h^0$ is chosen as the elliptic projection at $t=0.$
\end{corollary}

In addition, if we assume that the mesh is quasi-uniform and the error in the elliptic projection $\eta$
for $\delta \geq 1$ 
satisfies
$$\norm{\eta}_{L^\infty(\Omega)} \leq C h^{r+1}\|\psi\|_{W^{2+\delta,\infty}(\Omega)}.$$
Then, a use of superconvergence result \eqref{eq:supercgt} with the discrete Sobolev inequality,
\begin{equation*}
    \norm{\zeta}_{L^{\infty}(\Omega)} \le C | \log h|^{1/2}  \norm{\nabla\zeta},
\end{equation*}
shows the following maximum norm estimate for $\delta \geq 1$
\begin{equation*}
    \norm{(\psi-\psi_h)(t)}^2_{L^{\infty}(\Omega)}\leq Ch^{2(r+1)}e^{-2\alpha t+\hat{C}} | \log h| \left( \|\psi\|_{W^{2+\delta,\infty}(\Omega)}^2 + \int_0^te^{2\alpha s}(\norm{\psi}_{2+\delta}^2+\norm{\partial_t\psi}_{2+\delta}^2)\ds\right).
\end{equation*}

\begin{remark}
    From Lemma \ref{lem:apriori_zeta_H1}, we arrive using the inverse inequality at optimal $H^2$ estimate only for $\delta\geq 1$ as
    $$\|\Delta \zeta (t)\| \leq C h^{-1}  \|\nabla \zeta(t)\|.$$
    Therefore, a use of triangle inequality with estimate of $\eta$ in $H^2$ norm from Lemma~\ref{lem:apriori_eta} shows for $\delta\geq 1$
    \begin{equation*}
        \norm{(\psi-\psi_h)(t)}^2_2 \leq C e^{-2\alpha t+\hat{C}}h^{2 r} \left(\norm{\psi_0}_{2+\delta}^2+ \int_0^te^{2\alpha s}(\norm{\psi}_{2+\delta}^2+\norm{\partial_t\psi}_{2+\delta}^2)\ds\right).
    \end{equation*}
\end{remark}
\begin{remark}
    When $F=0$ or $F$ exponentially decays in time, the error decays exponentially in time.
    Moreover, if $F\in L^2(L^2),$ then from the regularity results, it is easy check that
    $\hat{C}=C(M,\nu)\int_0^t|\psi|_2^2\ds \leq C,$ that is, $\hat{C}$ becomes a constant, independent of time. Therefore, error
    analysis is valid uniformly in time. However, if $F\in L^{\infty}(L^2),$ then $\hat{C}\leq C\;t$  and $e^{Ct}$ term appears in the error analysis, making it local. Like in Navier-Stokes equations, it may be possible to prove the uniform validity of the  error
 estimates  with respect to time under the smallness assumption of the data.
\end{remark}

\section{Backward Euler Method}\label{sec:back}
This section discusses a fully discrete scheme based on the backward Euler method applied to the semi-discrete problem and derives its convergence analysis.

Let $\Delta t$ be the time step and $t_n=n\Delta t$, $n=0,1,2,\ldots,N$.
For a continuous function $\varphi$ on time, let $\ddt\varphi(t_n)=\frac{\varphi(t_n)-\varphi(t_{n-1})}{\Delta t}$.
Then, the backward Euler scheme is to find $\Psi^n\in S_h^0$, $n=1,2,\ldots,N$ such that
\begin{equation}\label{eq:fdiscprob}
    \einner{\nabla \ddt\Psi^n,\nabla\chi}+\nu a(\Psi^n,\chi)+b(\Psi^n;\Psi^n,\chi)+\mu b_0(\Psi^n,\chi)=\mu\einner{F,\chi}\quad\forall\chi\in S_h^0
\end{equation}
where $\Psi^0=\psi_{0,h}\in S_h^0.$ 
Since $S_h^0$ is finite dimensional, at each time level $t=t_n$, \eqref{eq:fdiscprob} leads to a system of nonlinear algebraic equations.
Therefore, we need to discuss existence and uniqueness result for the fully discrete system.

\subsection{Uniform \textit{a priori} bounds.}
This subsection is on \textit{a priori} bounds of the discrete solution, which are valid uniformly in time.
\begin{lemma}\label{lem:gronwalltype}
    With $0\leq \alpha<\frac{\nu\lambda_1}{4}$, choose $\Delta t>0$ so that $0<\Delta t<\Delta t_0$ satisfying $(\frac{\nu \Delta t\lambda_1}{4}+1)>e^{\alpha \Delta t}$.
Then for $N\geq 1$
    \begin{equation*}
        \norm{\nabla \Psi^N}^2+2\beta^{*}e^{2\alpha t_N}\Delta t\sum_{n=1}^Ne^{2\alpha t_n}\norm{\Delta \Psi^n}^2\leq e^{-\alpha t_N}\norm{\nabla\Psi^0}^2+\frac{e^{-\alpha \Delta t}}{2\alpha \nu} \big(1- e^{-2\alpha t_N}\big)\;\norm{F}_{-2}^2,
    \end{equation*}
    where $\beta^*=\big(e^{-\alpha \Delta t} \frac{\nu}{2}-\frac{1-e^{-\alpha \Delta t}}{\Delta t\lambda_1}\big)>0.$
\end{lemma}
\begin{proof}
    Choose $\chi=e^{2\alpha t_n}\Psi^n$ in \eqref{eq:fdiscprob}.
Note that $b(\Psi^n;\Psi^n,\Psi^n)=0$, $b_0(\Psi^n,\Psi^n)=0$ and
    \begin{equation}\label{eq:fdisc-1}
        e^{\alpha t_n}\ddt\Psi^n=e^{\alpha \Delta t}\ddt(\hat\Psi^n)-\frac{e^{\alpha \Delta t}-1}{\Delta t}\hat\Psi^n.
    \end{equation}
    Then after multiplying by $e^{-\alpha \Delta t}$ and with $\hat\Psi^n=e^{\alpha t_n}\Psi^n$, it follows that
    \begin{equation*}
        (\ddt\hat\Psi^n,\hat\Psi^n)+\nu\alpha^{-\alpha \Delta t}a(\hat\Psi^n,\hat\Psi^n)-\left(\frac{1-e^{-\alpha \Delta t}}{\Delta t}\right)\norm{\nabla\hat\Psi^n}^2=\mu e^{-\alpha \Delta t}(e^{\alpha t_n}F,\hat\Psi^n).
    \end{equation*}
    Observe that
    \begin{equation}\label{eq:fdisc-3}
        (\ddt \nabla\hat\Psi^n,\nabla\hat\Psi^n)\geq \frac{1}{2\Delta t}\ddt \norm{\nabla \hat\Psi^n}^2.
    \end{equation}
    Apply the Poincar\'e inequality to arrive at
    \begin{equation*}
        \frac{1}{2}\ddt \norm{\nabla\hat\Psi^n}^2+\big(e^{-\alpha \Delta t}\nu-\frac{1-e^{-\alpha \Delta t}}{\Delta t}\frac{1}{\lambda_1}\big)|\hat\Psi^n|_2^2\leq \frac{\mu e^{-\alpha \Delta t}}{2\nu}\norm{e^{\alpha t_n}F}_{-2}^2+\frac{\nu}{2}e^{-\alpha \Delta t}|\hat\Psi^n|_2^2.
    \end{equation*}
    With $\beta^*=\big(e^{-\alpha \Delta t}\frac{\nu}{2}-\frac{1-e^{-\alpha \Delta t}}{\Delta t\lambda_1}\big)>0$, take summation with respect to $n$ up to $N$,  and use $\Delta \sum_{n=0}^{N} e^{2\alpha t_n} \leq \frac{1}{2\alpha} (e^{2\alpha t_N} -1).$ Then, multiply the resulting equation by $e^{-\alpha t_N}$
to obtain
    \begin{equation*}
        \norm{\nabla\Psi^N}^2+2\beta^*e^{-\alpha t_N}\sum_{n=1}^Ne^{2\alpha t_n}|\Psi^n|_2^2\leq e^{-\alpha t_N}\norm{\nabla\Psi^0}^2+\frac{e^{\alpha \Delta t}}{2\alpha\nu} \big(1- e^{-2\alpha t_N}\big)\;  \norm{F}_{-2}^2.
    \end{equation*}
    This completes the rest of the proof.
\end{proof}
In order to estimate $\lim_{N\rightarrow\infty}|\Delta\Psi^N|^2\leq \frac{\mu^2}{\nu^2}\norm{F}_{-2}^2$, following Pany et al. \cite{pany2017}, we shall apply the following counterpart of the L'Hospital rule.
For a proof, see pp. 85--89 of \cite{muresan2009}.
\begin{theorem}{(Stolz-Cesaro Theorem).}
    Let $\{\varphi^n\}_{n=0}^\infty$ be a sequence of numbers, and let $\{w^n\}_{n=0}^N$ be a strictly monotone and divergent sequence.
    If
    \begin{equation*}
        \lim_{n\rightarrow\infty}\left(\frac{\varphi^n-\varphi^{n-1}}{w^n-w^{n-1}}\right)=\ell,
    \end{equation*}
    then
    \begin{equation*}
        \lim_{n\rightarrow\infty}\left(\frac{\varphi^n}{w^n}\right)=\ell.
    \end{equation*}
\end{theorem}
In the estimate of Lemma~\ref{lem:gronwalltype}, now set $\varphi^N=\frac{\nu}{2}e^{-\alpha \Delta t}\Delta t\sum_{n=0}^N|\Delta\Psi^n|_2^2$ and $w^N=e^{2\alpha t_N}$.
Then, an appeal to the Scholz-Cesaro theorem yields
\begin{equation*}
    \frac{\nu}{2(1-e^{-2\alpha \Delta t)}}e^{-\alpha \Delta t}\Delta t\limsup_{N\rightarrow\infty}|\Psi^N|_2^2\leq \Delta t\frac{N^2}{\nu(1-e^{-\alpha \Delta t})}\norm{F}_{-2}^2,
\end{equation*}
and hence,
\begin{equation*}
    \limsup_{N\rightarrow\infty}|\Psi^N|_2^2\leq \frac{\mu^2}{\nu^2}\norm{F}_{-2}^2.
\end{equation*}
As a consequence, it is possible to find some large $N_0\in\mathbb{N}$ so that for $N\geq N_0$
\begin{equation*}
    |\Psi^N|_2\leq \frac{\mu}{2 \nu}\norm{F}_{-2} <C.
\end{equation*}

\subsection{Wellposedness and existence of discrete attractor.}

For its wellposedness, we appeal to the following variant of the Brouwer fixed point theory.
For proof, see Kesavan \cite{kesavan1989}.
\begin{lemma}\label{lem:brouwer}
    Let $X$ be a finite dimensional Hilbert space with inner product $(\cdot,\cdot)_X$ and norm $\norm{\cdot}_X$.
Further, let $\mathcal{F}$ be a continuous map from $X$ into itself such that $(\mathcal{F}(\chi),\chi)>0$ for all $\chi\in X$ with $\norm{\chi}_X=R>0$.
Then, there exists $\chi^\star\in X$ with $\norm{\chi^\star}_X\leq R$ such that $\mathcal{F}(\chi^\star)=0$.
\end{lemma}
\begin{theorem}\label{thm:uniqueness}
    Assume that $\Psi^0,\cdots,\Psi^{n-1}$ are given.
    Then there exists a unique solution $\Psi^n$ of \eqref{eq:fdiscprob}.
\end{theorem}
\begin{proof}
    In order to apply Lemma~\ref{lem:brouwer}, set $X=S_h^0$ and define $\mathcal{F}$ as
    \begin{equation}\label{eq:solOp}
        \begin{aligned}
        (\mathcal{F}(\Phi),\chi)
        &=(\nabla\Phi,\nabla\chi)+\nu \Delta ta(\Phi,\chi)+\Delta tb(\Phi;\Phi,\chi)\\
        &\quad+\mu \Delta t b_0(\Phi,\chi)-\mu \Delta t(F^n,\chi)-(\nabla \Psi^{n-1},\nabla\chi).
        \end{aligned}
    \end{equation}
    Choose $\chi=\Phi$ in \eqref{eq:solOp}, then with $b(\Phi;\Phi,\Phi)=0$ and $\mu(\partial_x\Phi,\phi)=0$, we arrive at
    \begin{equation*}
        \begin{aligned}
            (\mathcal{F}(\Phi),\Phi) & \geq \norm{\nabla\Phi}^2+\nu \Delta t|\Phi|_2^2-(\mu \Delta t\norm{F^n}_{-1}+\norm{\nabla\Psi^{n-1}})\norm{\nabla\Phi}_{2} \\
                                     & \geq \norm{\nabla\Phi}\big(\norm{\nabla\Phi}-(\mu \Delta t\norm{F}_{-1}+\norm{\nabla\Psi^{n-1}})\big).
        \end{aligned}
    \end{equation*}
    Now choose $\norm{\nabla\Phi}\geq 2(\mu\norm{F}_{-1}+\norm{\nabla\Psi^{n-1}})=R$ so that $(\mathcal{F}(\Phi),\Phi)>0$.
    Therefore, an application of Lemma~\ref{lem:brouwer} yields existence of $\Psi^n\in S_h^0$ such that $\mathcal{F}(\Psi^n)=0$.
    For uniqueness, let $\Psi^n_1$ and $\Psi^n_2$ be two solutions of \eqref{eq:fdiscprob}.
    Setting $\Phi^n=\Psi^n_1-\Psi^n_2$, it satisfies
    \begin{equation}\label{eq:uniqueerr}
        (\nabla\ddt\Phi^n,\nabla\chi)+\nu a(\Phi^n,\chi)+\mu  b_0(\Phi^n,\chi)= -\big(b(\Psi^n_1;\Psi^n_1,\chi)-b(\Psi^n_2;\Psi^n_2,\chi)\big).
    \end{equation}
    Choose $\chi=\Phi^n$ in \eqref{eq:uniqueerr}, then we obtain
    \begin{equation}
        \frac{1}{2}\ddt\norm{\nabla\Phi^n}^2+\nu|\Phi^n|_2^2\leq b(\Psi_1^n;\Psi_2^n,\Phi^n)-b (\Psi^n_1;\Psi^n_1,\Phi^n)=-b(\Psi^n_2;\Phi^n,\Phi^n)-b(\Phi^n;\psi^n_1,\Phi^n).
    \end{equation}
    Since $b(\Psi_2^n;\cdot,\cdot)$ is skew-symmetric, it remains to estimate the last term on the right-hand side.
    \begin{equation*}
        \begin{aligned}
            |-b (\Phi^n;\Psi^n_1,\Phi^n)|
             & \leq M|\Phi^n|_2^{3/2}\norm{\nabla\Psi_1^n}^{1/2}|\Psi^n_1|_2^{1/2}\norm{\nabla\Phi^n}^{1/2} \\
             & \leq \frac{\nu}{2}|\Phi^n|_2^2+C(M,\nu)\norm{\nabla\Phi^n}^2|\Psi^n_1|_2^2\norm{\nabla\Phi^n}^2 \\
             & \leq \frac{\nu}{2}|\Phi^n|_2^2+C(M,\nu)|\Psi^n_1|_2^2\norm{\nabla\Phi^n}^2.
        \end{aligned}
    \end{equation*}
    Here, we have used $ab\leq a^p/p+ b^q/q$ and boundedness of $\norm{\nabla\Psi^n_1}$.
    On substitution, summation on $n$ from $n=1$ to $N$ with $\Phi^0=0$ yields
    \begin{equation*}
        (1- C \Delta t |\Psi_1^N|_2^2) \norm{\nabla\Phi^N}^2+\nu \Delta t\sum_{n=1}^N|\Phi^N|_2^2\leq C\Delta t\sum_{n=1}^{N-1} |\Psi_1^n|_2^2\norm{\nabla\Phi^n}^2.
    \end{equation*}
    Choose $N$ large so that $|\Phi^N|_2\leq C$.
Then, for small $\Delta t$, $(1-C\Delta t)=1/2$ and we now arrive at
    \begin{equation*}
        \norm{\nabla\Phi^N}^2+2\nu \Delta t\sum_{n=1}^N|\Phi^n|_2^2\leq C\Delta t\sum_{n=1}^{N-1}|\Psi_1^n|_2^2\norm{\Phi^n}^2.
    \end{equation*}
    Apply the discrete Gronwall inequality to infer $\norm{\nabla\Phi^N}=0$, that is $\Phi^N=0$ which implies uniqueness.
\end{proof}
\begin{remark}
    From Theorem~\ref{thm:uniqueness}, given $\Psi^{n-1}\in S_h^0$, there is a unique solution $\Psi^n\in S_h^0$ which in turn, defines a map $\mathcal{S}^n:S_h^0\rightarrow S_h^0$ such that $\mathcal{S}^n(\Psi^{n-1})=\Psi^n$ which is continuous.
\end{remark}
As a consequence, we obtain the following result.
\begin{corollary}
    There exists a bounded absorbing set
    \begin{equation*}
        B^h_{\rho_0}(0)\coloneq \{v\in S_h^0:\norm{\nabla v}\leq \rho_0\}
    \end{equation*}
    where
    \begin{equation*}
        \rho_0^2=\frac{e^{2\alpha \Delta t}}{\alpha\nu}\norm{F}_{-2}^2,
    \end{equation*}
    in the sense that for $\rho>0$, there is $t_{n^*} (\rho)>0$ such that for all $t_N\geq t_{n^*}$,
    $${\mathcal{S}}^N(t_N) B^h_{\rho} \subset \mathcal{S}^N(t_N) B^h_{\rho_0}.$$
\end{corollary}
The proof goes parallel to the continuous case.
For simplicity, we indicate its proof briefly.
\begin{proof}
    We now claim that if $\Psi^0\in B_{\rho_1}(0)$ for some $\rho_1>\rho_0/2>0$, then there exists $t_{n^*}=n^*\Delta t$ depending on $\rho_1$ such that for $t_N\ge_{n^*}$, the discrete solution $\Psi^N$ lies in $B_{\rho_0}(0)$.
    Observe from Lemma~\ref{lem:gronwalltype} that
    \begin{equation}
        \norm{\nabla\Psi^N}\leq e^{-2\alpha t_N}\norm{\nabla\Psi^0}+\frac{\rho_0^2}{2}.
    \end{equation}
    To complete the first part of the proof, it is sufficient to show that
    \begin{equation*}
        e^{-\alpha t_N}\norm{\nabla\Psi^0}\leq \frac{\rho_0^2}{2}
    \end{equation*}
    and this holds provided there is $t_{n^*}=n^*\Delta t\geq \frac{1}{\alpha}\log\left(\frac{\rho_1}{\rho_0}\right)$.
    For $\rho_1<\rho_0/2$, the result holds trivially.
    This completes the rest  of the proof.
\end{proof}

\subsection{\textit{A priori} error analysis of the fully discrete scheme.}
This subsection focuses on optimal error estimates of the fully discrete scheme.

Since $\psi(t_n)-\Psi^n= (\psi(t_n) - \psi_h(t_n)) -(\Psi^n-\psi_h(t_n))$ and the estimate of $\psi(t_n)-\psi_h(t_n) $ is known from Section 4, it is now enough to estimate $\bm{e}_h^n:=(\Psi^n-\psi_h(t_n)).$
Now from the semi-discrete problem \eqref{eq:discprob} and the fully discrete scheme
\eqref{eq:discprob}, we obtain with $\psi_h(t_n) =\psi_h^n$ an equation in $\bm{e}_h^n$ for $n=1,2,\ldots,$ as
\begin{equation}\label{eq:error-e-n}
    \begin{aligned}
    (\nabla \ddt \bm{e}_h^n,\nabla\chi)&+\nu a( \bm{e}_h^n,\chi)
    =(\nabla\tau^n, \nabla\chi)\\
     &- \Big(b(\psi_h^n;\psi_h^n, \chi) - b(\Psi^n;\Psi^n,\chi)\Big)
    - \mu b_0( \bm{e}_h^n,\chi)\quad\forall \chi\in S_h^0,
    \end{aligned}
\end{equation}
where
\begin{equation}\label{eq:def-tau-n}
    \tau^n=\partial_t\psi_{h}(t_n)-\bar\partial_t \psi_h^n = \frac{1}{2\Delta t}\displaystyle{\int_{t_{n-1}}^{t_n}}(t-t_n)  \partial_t^2\psi_{h}(t) dt.
\end{equation}
Below, we state and prove the main theorem of this section.
\begin{theorem} \label{thm:error-fully-discrete}
    Let $0\leq \alpha<\frac{\nu\lambda_1}{4}$, and let $\Delta t>0$ satisfies $(\frac{\nu \Delta t\lambda_1}{4}+1)>e^{\alpha \Delta t}$.
    Then, there exists a positive constant
    $C$, independent of $\Delta t $, such that for $n=1,2,\cdots $
    \begin{equation*}
        \|\nabla \bm{e}_h^n\|^2+ \beta_1 \Delta t e^{-2\alpha t_n}\sum_{i=1}^{n}e^{2\alpha t_i} |\nabla \bm{e}_h^i|_2^2
        \leq C (t_N)   (\Delta t)^2.
    \end{equation*}
\end{theorem}

\begin{proof}
    Set $\chi=e^{2\alpha t_n}  \bm{e}_h^n$ in \eqref{eq:error-e-n}.
    We use \eqref{eq:fdisc-1} and \eqref{eq:fdisc-3} by replacing $\Psi^n$ by $\bm{e}_h^n$ and skew-symmetricity of $b_0(\cdot,\cdot)$ to obtain
    \begin{align} \label{eq:fdisc-4}
        &\frac{1}{2}\ddt \norm{\nabla\hat{\bm{e}}_h^n}^2+\big(e^{-\alpha \Delta t}\nu-\frac{1-e^{-\alpha \Delta t}}{\Delta t}\frac{1}{\lambda_1}\big) \|\Delta \hat{\bm{e}}_h^n\|^2\\
        &\;\;\;\;\;\;\leq e^{-\alpha \Delta t}  \|\nabla \hat{\tau}^n\|  \|\nabla \hat{\bm{e}}_h^n\| + e^{-\alpha \Delta t}  |b(\hat{\bm{e}}_h^n;\psi_h^n, \hat{\bm{e}}_h^n)| \nonumber\\
        &\;\;\;\;\;\;\leq C e^{-\alpha \Delta t} \Big(\|\nabla \hat{\tau}^n\|  \|\Delta \hat{\bm{e}}_h^n\|
        + \|\Delta \hat{\bm{e}}_h^n\| \|\nabla \psi^n_h\|_{L^4}  \|\nabla \hat{\bm{e}}_h^n\|_{L^4}\Big)\nonumber\\
        &\;\;\;\;\;\;\leq C e^{-\alpha \Delta t} \Big(\|\nabla \hat{\tau}^n\|  \|\Delta \hat{\bm{e}}_h^n\|
        + \|\Delta \hat{\bm{e}}_h^n\|^{3/2}  \|\Delta \psi_h^n\|^{1/2} \|\nabla \psi^n_h\|^{1/2}  \|\nabla \hat{\bm{e}}_h^n\|^{1/2}\Big)\nonumber\\
        &\;\;\;\;\;\;\leq C e^{-\alpha \Delta t} \Big(\|\nabla \hat{\tau}^n\|^2
        + \|\Delta \psi_h^n\|^2 \|\nabla \psi^n_h\|^2  \|\nabla \hat{\bm{e}}_h^n\|^2\Big) + \frac{\nu}{2}  e^{-\alpha \Delta t} \|\Delta \hat{\bm{e}}_h^n\|^2
    \end{align}
    where $\hat{v}^n=e^{\alpha t_n}v^n$.
    Here, we used the Poincar\'e inequality, generalized H\"older's inequality and Young's inequality as in Lemma~\ref{lem:gronwalltype}.
    On summing up \eqref{eq:fdisc-4} from $n=1$ to $N$, we obtain with $\bm{e}_h^0=0$
    \begin{align} \label{eq:fdisc-5}
        \norm{\nabla\hat{\bm{e}}_h^N}^2+ 2 \beta^{\star}
        \Delta t \sum_{n=1}{N} \|\Delta \hat{\bm{e}}_h^n\|^2 &\leq C \Delta t \sum_{n=1}^{N}  e^{2\alpha t_{n-1} }\|\nabla {\tau}^n\|^2 \nonumber\\
        &+ C e^{-\alpha \Delta t}  \Delta t \sum_{n=1}^{N} \|\Delta \psi_h^n\|^2 \|\nabla \psi^n_h\|^2  \|\nabla \hat{\bm{e}}_h^n\|^2.
    \end{align}
    From the definition \eqref{eq:def-tau-n}, it follows that
    \begin{equation}\label{eq:estimate-tau-n}
        e^{2\alpha t_{n-1} } \|\nabla {\tau}^n\|^2 \leq \frac{1}{(\Delta t)^2} e^{2\alpha t_{n-1} } \Big( \int_{t_{n-1}}^{t_n} \|\nabla \partial^2_t \psi_h(s) \|\ds\Big)^2 \leq \frac{\Delta t}{3}  \int_{t_{n-1}}^{t_n} e^{2\alpha s} \|\nabla \partial^2_t \psi_h(s) \|^2\ds.
    \end{equation}
    Substitute \eqref{eq:estimate-tau-n} in \eqref{eq:fdisc-5} to arrive with smallness of $\Delta t$ at
    \begin{align*}
        \norm{\nabla\hat{\bm{e}}_h^N}^2+ 2 \beta^{\star}
        \Delta t \sum_{n=1}{N} \|\Delta \hat{\bm{e}}_h^n\|^2 &\leq C (\Delta t)^2  \int_{0}^{t_N}  e^{2\alpha s}
        \|\nabla \partial^2_t \psi_h(s) \|^2\ds \nonumber\\
        &+ C e^{-\alpha \Delta t}  \Delta t \sum_{n=1}^{N-1} \|\Delta \psi_h^n\|^2 \|\nabla \psi^n_h\|^2  \|\nabla \hat{\bm{e}}_h^n\|^2.
    \end{align*}
    An application of the Gronwall's Lemma now yields
    \begin{equation*}
        \begin{aligned}
        &\norm{\nabla\hat{\bm{e}}_h^N}^2 + 2 \beta^{\star}
        \Delta t \sum_{n=1}{N} \|\Delta \hat{\bm{e}}_h^n\|^2\\
        &\;\;\;\;\;\;\leq C (\Delta t)^2 \Big(\int_{0}^{t_N}  e^{2\alpha s} \|\partial^2_t \nabla \psi_h(s) \|^2\ds\Big)   \exp \Big(\Delta t \sum_{n=0}^{N-1} \|\Delta \psi_h^n\|^2 \|\nabla \psi^n_h\|^2\Big)\\
        &\;\;\;\;\;\;\leq C (\Delta t)^2  \Big(\int_{0}^{t_N}  e^{2\alpha s} \|\nabla \partial^2_t \psi_h(s) \|^2\ds\Big)   \exp \Big(\int_{0}^{t_N} \|\Delta \psi_h(s)\|^2 \|\nabla \psi_h(s)\|^2\ds\Big).
        \end{aligned}
    \end{equation*}
    This completes the rest of the proof.
\end{proof}
\section{Numerical Experiments}\label{sec:num}

This section focusses on several numerical experiments on benchmark problems and then on their  results confirming our theoretical findings.

For discrete conforming space, we choose Hsieh-Clough-Tocher (HCT) macro-element which is a piecewise cubic polynomial space on each sub-triangle.
\subsection{Convergence rate}
Since we consider the backward Euler method, the time discretization error yields $\mathcal{O}(\Delta t)$.
Therefore, the expected optimal error for a smooth function is, for $s=0,1,2$,
\begin{equation*}
    \|\psi(T)-\psi_h^N\|_s\leq C(h^{4-s}+\Delta t).
\end{equation*}
To verify convergence order of space discretization error, we compare error at the final time measured in $L^2$-, $H^1$-, and $H^2$-norm with varying $h$ and fixed $\Delta t=10^{-3}$.
Here, $\Delta t$ is chosen small so that the error from time discretization can be neglected.
The tested solution is defined by
\begin{equation*}
    \psi = \frac{1-e^{-t}}{1-e^{-0.1}}\sin(\pi x)^2sin(\pi y)^2
\end{equation*}
with $\psi(0)\equiv 0$, $\nu= 1.6667$, and $\mu=10^{3}$.
$F$ is chosen so that the solution satisfies \eqref{eq:qge}.
Convergence history with respect to the degrees of freedom is reported in Figure~\ref{fig:conv-hist-smooth}.
\begin{figure}
    \centering
    \includegraphics[width=0.45\textwidth]{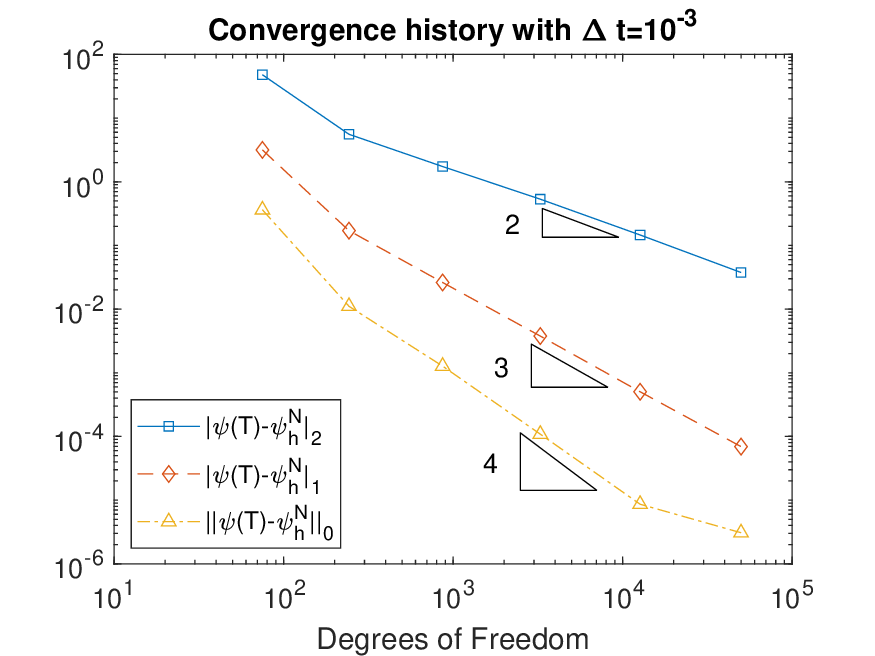}
    \caption{Convergence history with respect to the degrees of freedom. Right triangles indicate the ideal convergence order with respect to $h$.\label{fig:conv-hist-smooth}}
\end{figure}
We have $\|\psi(T)-\psi_h^N\|_j=\mathcal{O}(h^{4-j})$ for $j=0,1,2$ which comply with the results in Theorem~\ref{thm:apriori}.
\subsection{Exponential decaying with no force}
This subsection is on the verification of  the exponential decaying property of
the solution when $F=0$. Consider a closed rectangular basin,
$\Omega=(0,1)\times(-1,1)$. With initial condition
$\psi_0=\sin(\pi x)^2\sin(\pi y)^2$, we observe $\|\nabla
\psi_h(t)\|_{L^2(\Omega)}$ for $t=(0,0.1)$ with various $\nu$ and
$\mu$. The results with $h=2^{-7}$ and $\Delta t=10^{-3}$ are
reported in Figure~\ref{fig:zero_test}.

\begin{figure}[t]
    \centering
    \includegraphics[width=0.45\textwidth]{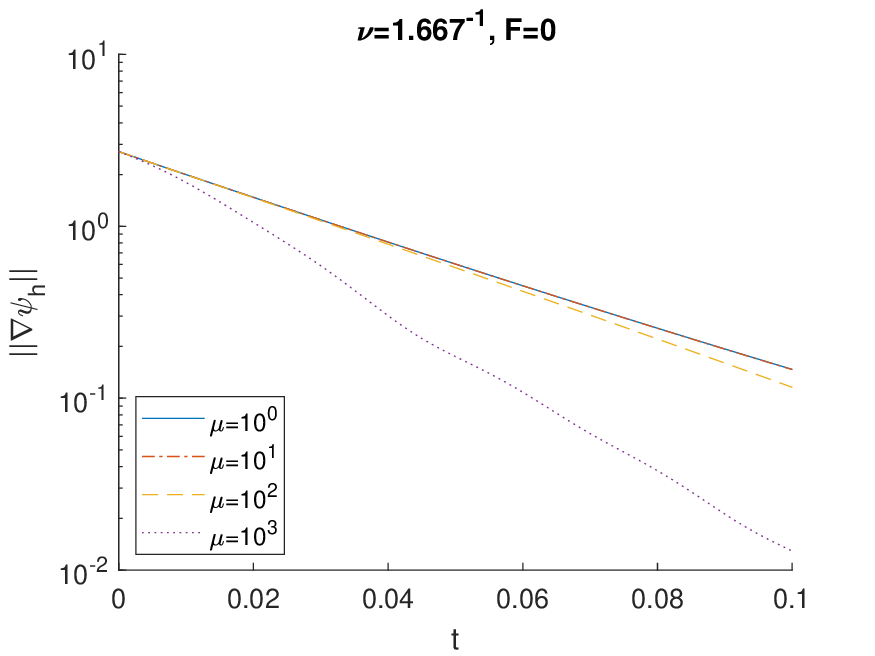}
    \includegraphics[width=0.45\textwidth]{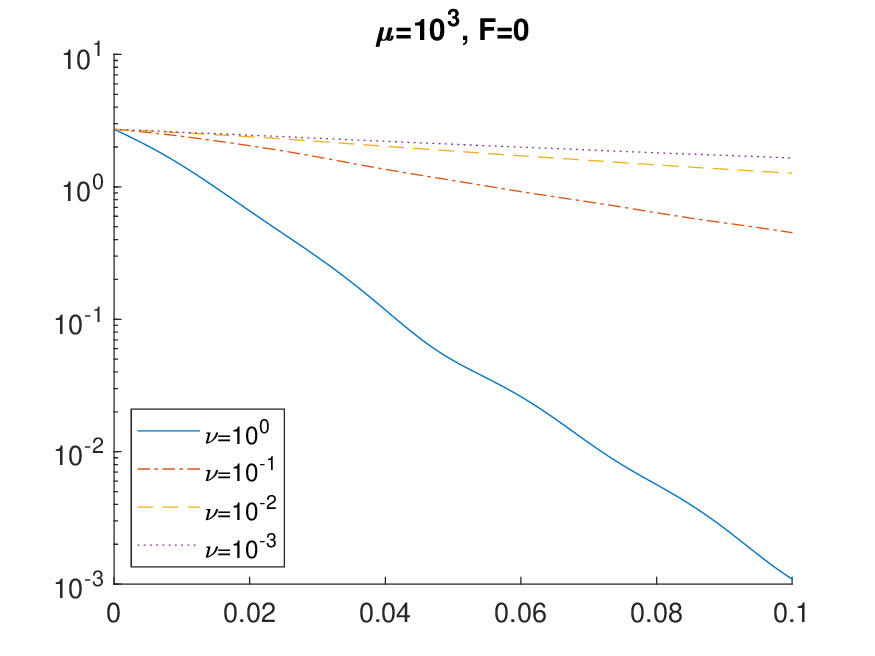}
    \caption{History of $\|\nabla\psi_h\|$ with various $\mu$ (left) and $\nu$ (right) when $F=0$.\label{fig:zero_test}}
\end{figure}
As expected in Remark~\ref{rmk:decaying}, the solution converges to 0 exponentially regardless of the choice of $\nu$ and $\mu$.
However, the rate of decay depends on $\nu$ and $\mu$.
Note that $\nu$ is the vorticity diffusion coefficient and $\mu$ is the convection speed.
Therefore, with homogeneous boundary condition, we can expect that the solution converges to 0 faster, when $\nu$ and $\mu$ are large. The numerical results conform with this heuristic.

\subsection{Attractor with time-independent force}
Finally, we observe the dynamics of stream function when $F$ is given as a non-trivial time-independent function.
For simplicity, we choose $\psi_0=0$ and $F=\sin(\pi y)$ which is a simplified wind shear stress, see \cite{bryan1963,sverdrup1947}.
\begin{figure}[t]
    \centering
    \includegraphics[width=0.45\textwidth]{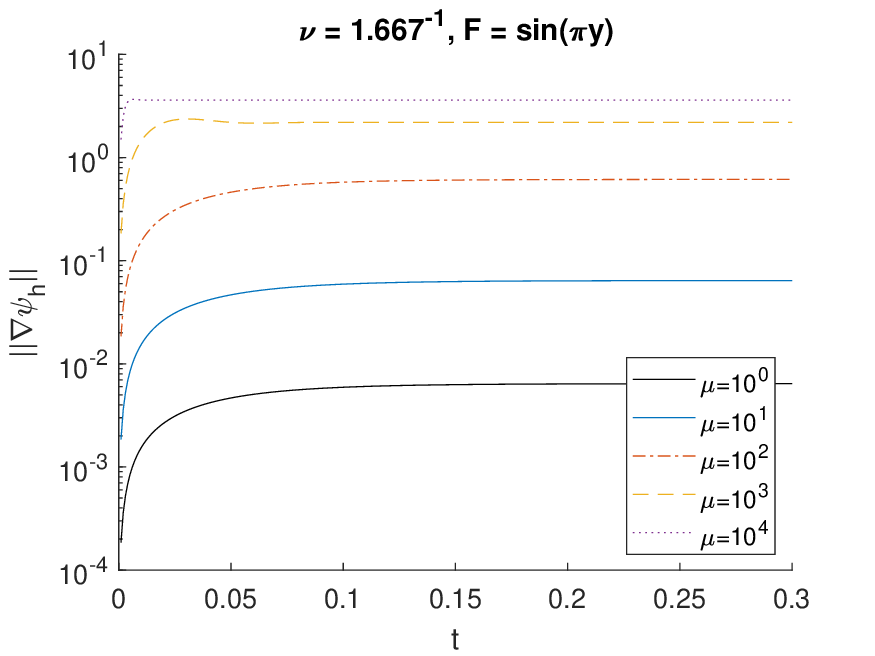}
    \includegraphics[width=0.45\textwidth]{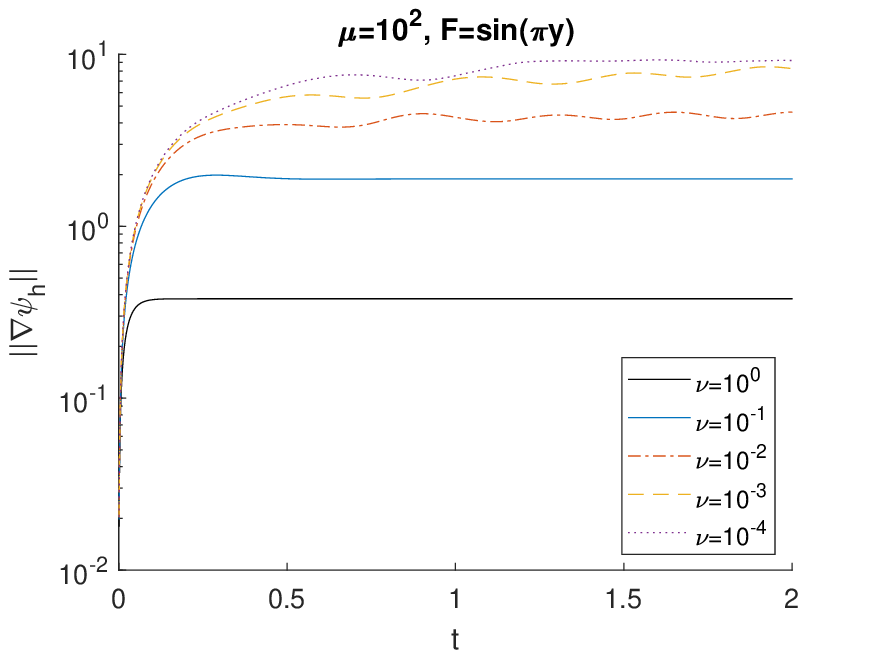}
    \caption{History of $\|\nabla\psi_h\|$ with various $\mu$ (left) and $\nu$ (right) when $F=\sin(\pi y)$.}
    \label{fig:sin_test}
\end{figure}
In Figure~\ref{fig:sin_test}, it is observed that the solution converges to a point attractor when $\nu$ is relatively large, i.e., relatively small $Re$.
On the other hand, when $\nu$ is relatively small, the energy locally fluctuates and does not converge to a stationary solution.
This indicates that solution transits from a laminar flow to a turbulent flow as $\nu$ varies.
Profiles of solution at $t=1,2,3,4$ are depicted in Figures~\ref{fig:profile_1}--\ref{fig:profile_3} for $\nu=1,\;0.01,\;0.0001$, respectively, with $\mu=100$.
The solution converges to a stationary solution when $\nu=1$ where the solution slightly skewed to the western boundary.
Also, the solution is almost symmetric with respect to the line $y=0$.
For $\nu=0.01$, the solution oscillates and does not converge to a stationary solution.
However, it still maintains symmetric behavior.
When $\nu=0.0001$, upper gyres and lower gyres are mixed and the solution shows chaotic behavior.

\begin{figure}[t]
    \centering
    \includegraphics[width=0.24\textwidth]{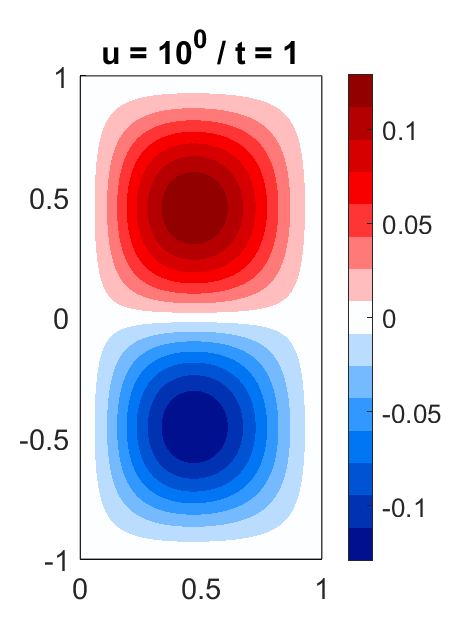}
    \includegraphics[width=0.24\textwidth]{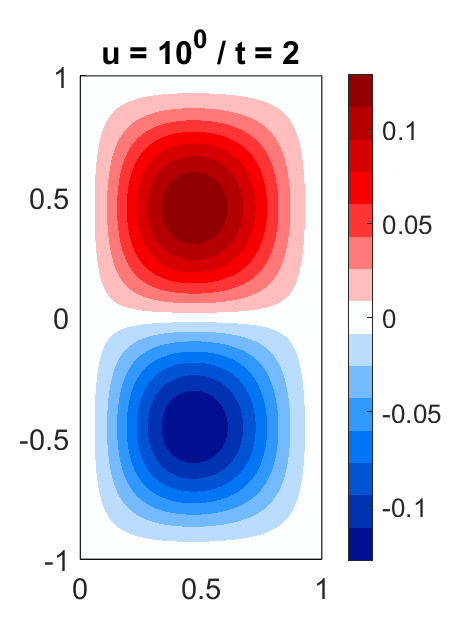}
    \includegraphics[width=0.24\textwidth]{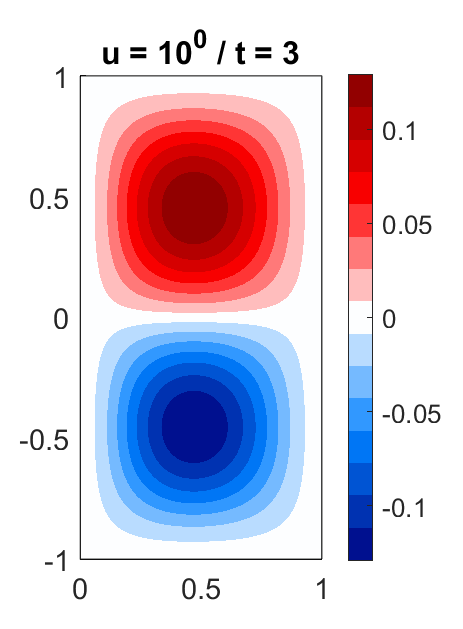}
    \includegraphics[width=0.24\textwidth]{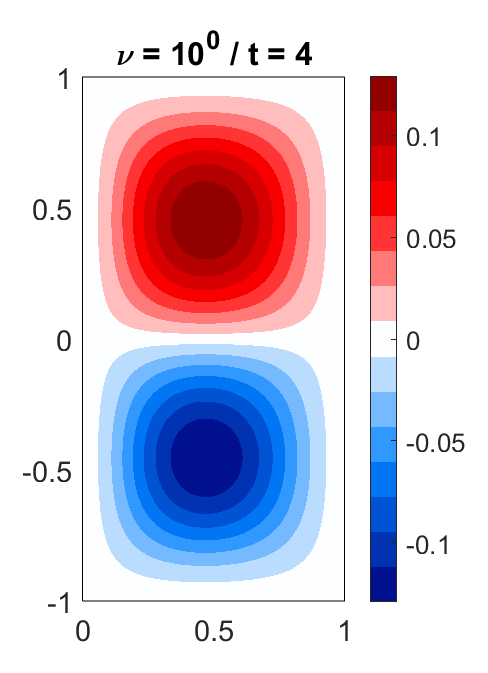}
    \caption{Streamfunction profiles at $t=1,2,3,4$ with $\nu=1$ and $\mu=100$.\label{fig:profile_1}}
\end{figure}
\begin{figure}[t]
    \centering
    \includegraphics[width=0.24\textwidth]{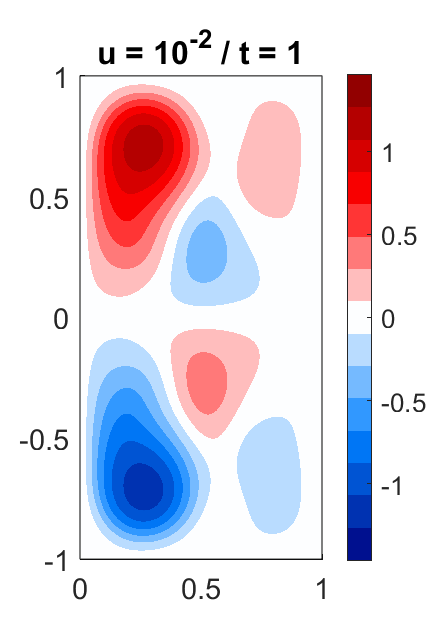}
    \includegraphics[width=0.24\textwidth]{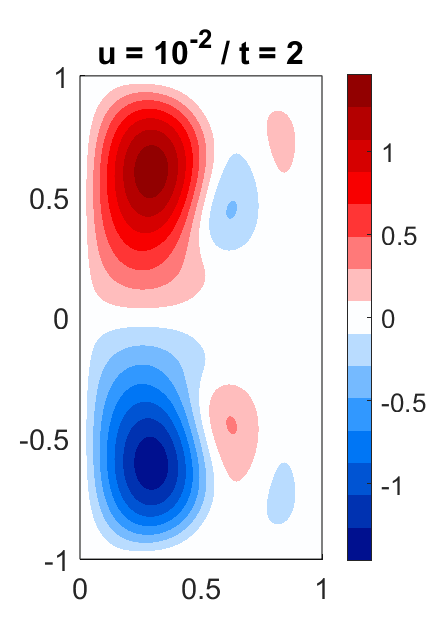}
    \includegraphics[width=0.24\textwidth]{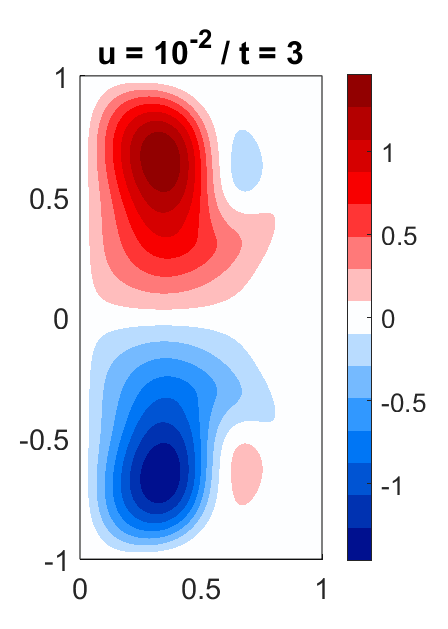}
    \includegraphics[width=0.24\textwidth]{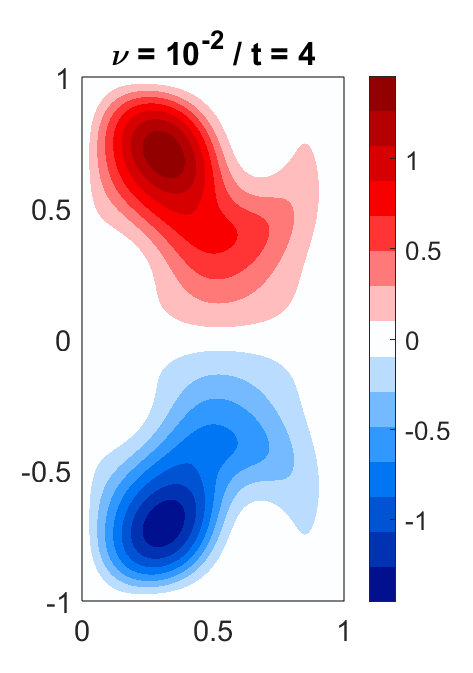}
    \caption{Streamfunction profiles at $t=1,2,3,4$ with $\nu=0.01$ and $\mu=100$.\label{fig:profile_2}}
\end{figure}
\begin{figure}[t]
    \centering
    \includegraphics[width=0.24\textwidth]{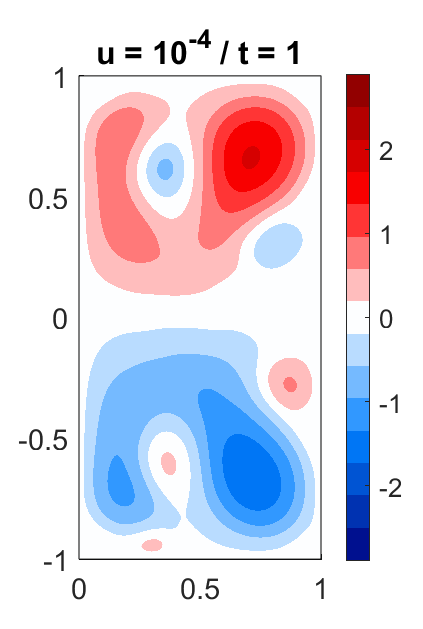}
    \includegraphics[width=0.24\textwidth]{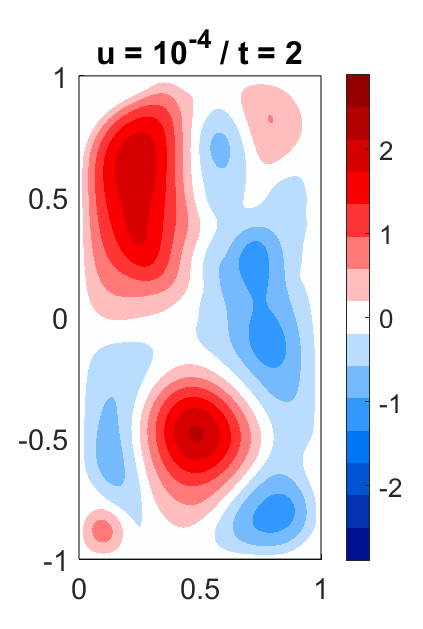}
    \includegraphics[width=0.24\textwidth]{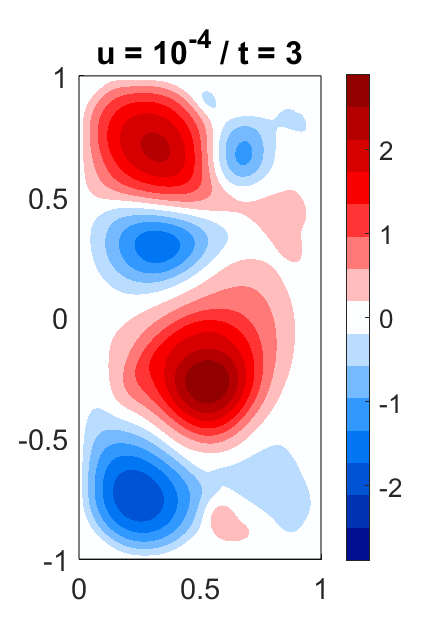}
    \includegraphics[width=0.24\textwidth]{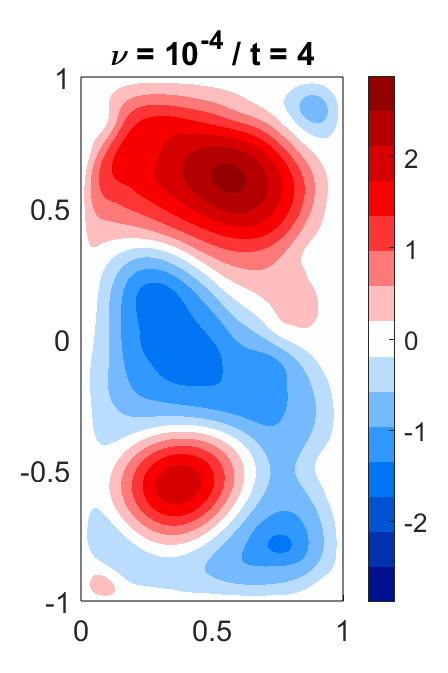}
    \caption{Streamfunction profiles at $t=1,2,3,4$ with $\nu=0.0001$ and $\mu=100$.\label{fig:profile_3}}
\end{figure}

\section{Conclusion}\label{sec:con}
In this study, we introduce a family of $C^1$-conforming finite element method for the evolutionary surface QG equation.
{\textit{A priori}} estimates for the continuous solution are provided.
The existence of a global attractor is established and remarks with some special cases, $F=0$ or time-independent, are provided.
The optimal convergence for $C^1$-conforming finite element method is derived, which also covers low regularity.
Similarly to the continuous case, the finite element solution also converges to a discrete attractor as $t\rightarrow\infty$.
The optimal convergence with a smooth solution is verified numerically with HCT-element.
The exponential decaying property with various choice of physical parameters is observed when $F=0$.
With non-zero time-independent $F$, the solution shows stable behavior with the bounded energy.
When $\nu$ is between $10^{-1}$ and $10^{-2}$, the energy fluctuates locally which shows that the solution becomes a turbulent flow rather than a laminar flow.

In the future, we extend our method to the multi-layer surface QG equations to cope with a vertically inhomogeneous flow.
Nonconforming finite element methods will be also considered to avoid the complicate implementation of $C^1$-finite element.
The theoretical findings in this study will result in the guidance of the future study.


\end{document}